\documentclass[preprint,12pt]{elsarticle}

\usepackage[papersize={210mm,297mm},lmargin=2.5cm,rmargin=2.5cm,top=3cm,bottom=3cm]{geometry}

\usepackage{amsmath,amsfonts,amssymb,amsthm}
\usepackage{latexsym,pdfsync,xcolor,graphicx}
\usepackage[nottoc]{tocbibind}
\usepackage[T1]{fontenc}
\usepackage[utf8]{inputenc}												
\usepackage[english]{babel}			
\usepackage{tikz}
\usepackage{appendix}
\usepackage{hyperref}
\usepackage{enumerate}
\usepackage[shortlabels]{enumitem}
\usepackage{longtable}
\usepackage{float}
\usepackage{multirow,array}
\usepackage{verbatim}
\usepackage{titlesec}
\usepackage{graphicx}
\usepackage{caption}
\usepackage{subcaption}
\usepackage{multicol}
\usepackage{stackengine}

\newtheorem{theorem}{Theorem}[section]

\newtheorem{lemma}[theorem]{Lemma}
\newtheorem{proposition}[theorem]{Proposition}
\newtheorem{remark}[theorem]{Remark}
\numberwithin{equation}{section}
\newtheorem*{conjecture*}{Conjecture}

\journal{****}

\begin{document}

\begin{frontmatter}

\title{Planar Kolmogorov systems with infinitely many singular points at infinity}
\author[a]{\'Erika Diz-Pita}
\address[a]{Departamento de Estat\'istica, An\'alise Matem\'atica e Optimizaci\'on, Universidade
	de Santiago de Compostela, 15782 Santiago de Compostela, Spain }
\ead{erikadiz.pita@usc.es}
\author[b]{Jaume Llibre}
\address[b]{Departament de Matem\`atiques, Universitat Aut\`onoma de Barcelona, 08193 Bellaterra, Barcelona, Spain}
\ead{jllibre@mat.uab.cat}
\author[a]{M. Victoria Otero-Espinar}
\ead{mvictoria.otero@usc.es}
\begin{abstract}
We classify the global dynamics of the five-parameter family of planar Kolmogorov systems
\begin{equation*}
	\begin{split}
		\dot{y}&=y \left( b_0+ b_1 y z + b_2 y + b_3 z\right),\\
		\dot{z}&=z\left( c_0 + b_1 y z + b_2 y + b_3 z\right),
	\end{split}
\end{equation*}
which is obtained from the Lotka-Volterra systems of dimension three.  These systems have infinitely many singular points at inifnity. We give the topological classification of their phase portraits in the Poincar\'e disc, so we can describe the dynamics of these systems near infinity. 
We prove that these systems have 13 topologically distinct global phase portraits.
\end{abstract}

\end{frontmatter}

\section{Introduction} 

Kolmogorov systems are differential systems of the form  
$$
\dot{x_i}=x_i  P_i(x_1,\ldots,x_n), \;\;\; i=1,...,n,
$$
where $P_i$ are polynomials. Particular cases of these systems are, for example, Lotka-Volterra systems. All of them have been used for modelling problems from different sciences as the interaction between species \cite{Arnoedo1980,Coste, LlibreLV2,CLois,Smale}, plasma physics \cite{plasma}, hydrodynamics \cite{hidrodinamica}, chemical reactions \cite{LVchemical} or economic and social problems \cite{ecoGandolfo,Palomba,ecoarticulo}.

For the Lotka-Volterra systems in dimension three the global dynamics has been described in some particular cases. In \cite{LlibreLV0}  the authors give the global phase portraits in the Poincaré disc of a system related with the study of black holes; in \cite{LlibreLV4} the authors complete the description of the global dynamics of a system previously proposed and studied in \cite{Leach,LlibreLV3,Tudoran}.
There are also some works about the global dynamics of certain Lotka-Volterra families depending on a small number of parameters. In \cite{LlibreLV5} the authors study a family depending on two parameters, and in \cite{LlibreLV1} the family studied depends on three parameters, but with some restrictions such as all of them must be positive. 

There are few works that study these kind of systems when they have a line consisting of singular points; an example can be found
in  \cite{Sch-Vul}. 

In \cite{SIS1} and \cite{SIS2}  the global dynamics of two Kolmogorov families in dimension two was studied. Those families are obtained from general $3$-dimensional Lotka-Volterra systems depending on 12 parameters, 
\begin{equation*}
	\begin{split}
		& \dot{x} = x \: ( \:  a_0 + a_1 x + a_2 y + a_3 z \: ),\\
		& \dot{y} = y \: ( \: b_0 + b_1 x + b_2 y + b_3 z \: ), \\
		& \dot{z} = z \: ( \: c_0 + c_1 x + c_2 y + c_3 z \: ),
	\end{split}
\end{equation*}
with a rational first integral of degree two of the form $x^i y^j z^k$ by applying the Darboux theory of integrability. For the obtained families, the condition that they have a Darboux invariant of the form $e^{st} y^{\lambda_1} z^{\lambda_2}$ is required. In this work we focus on the systems studied in \cite{SIS2} which are
\begin{equation}\label{SIS2}
	\begin{split}
		\dot{y}&=y \left( b_0+ b_1 y z + b_2 y + b_3 z\right),\\
		\dot{z}&=z\left( c_0- \mu ( b_1 y z + b_2 y + b_3 z)\right),
	\end{split}
\end{equation}
and depend on six parameters. In \cite{SIS2} the authors give the topological classification of the global phase portraits in the Poincaré disc for all the values of the parameters such that $\mu\neq-1$. 
The particular case with $\mu=-1$, in which there exists a line of singular points (more precisely, all the infinity consist on singular points) was not studied,
so here we carry out the study of this case, i.e. we deal with the systems
\begin{equation}\label{sis2mu}
	\begin{split}
		\dot{y}&=y \left( b_0+ b_1 y z + b_2 y + b_3 z\right),\\
		\dot{z}&=z\left( c_0 + b_1 y z + b_2 y + b_3 z\right).
	\end{split}
\end{equation}

In this paper we study the global dynamics of systems \eqref{sis2mu} and we give the topological classification of all their global phase portraits on the Poincaré disc. Our main result is the following.

\begin{theorem}\label{th_global}
	Cubic Kolmogorov systems \eqref{sis2mu} have $13$ topologically distinct phase portraits in the Poincar\'e disc, given in Figure $\ref{fig:global_top}$.
\end{theorem}

\begin{figure}[h] 
	\centering
	\scriptsize
	\stackunder[2pt]{\includegraphics[width=3cm]{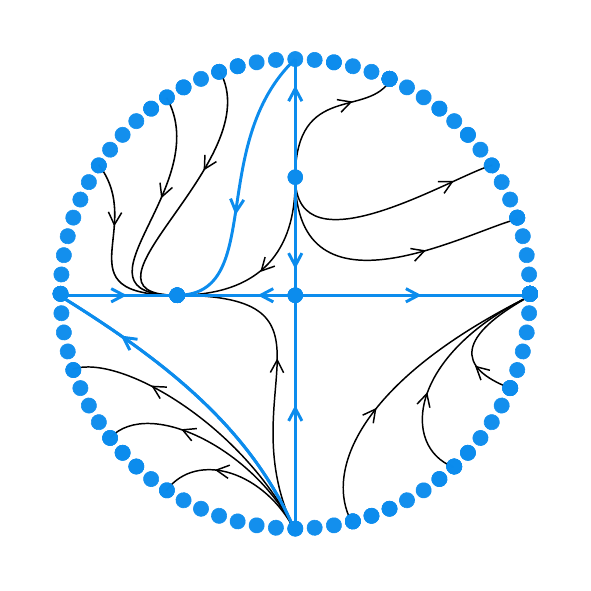}}{(R1)}
	\stackunder[2pt]{\includegraphics[width=3cm]{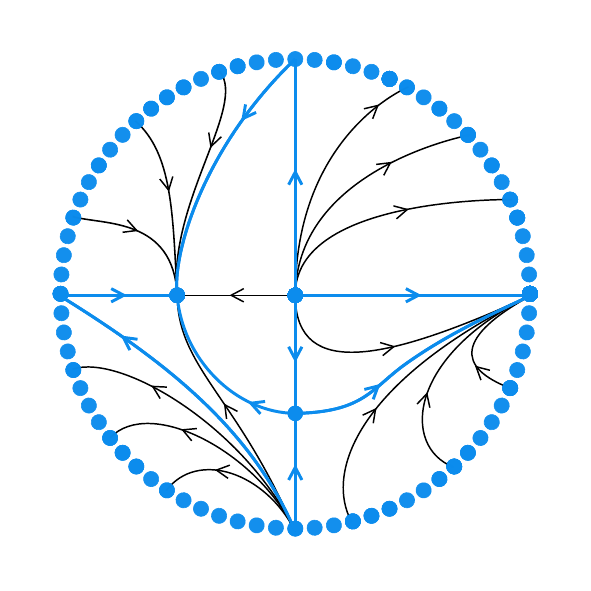}}{(R2)}
	\stackunder[2pt]{\includegraphics[width=3cm]{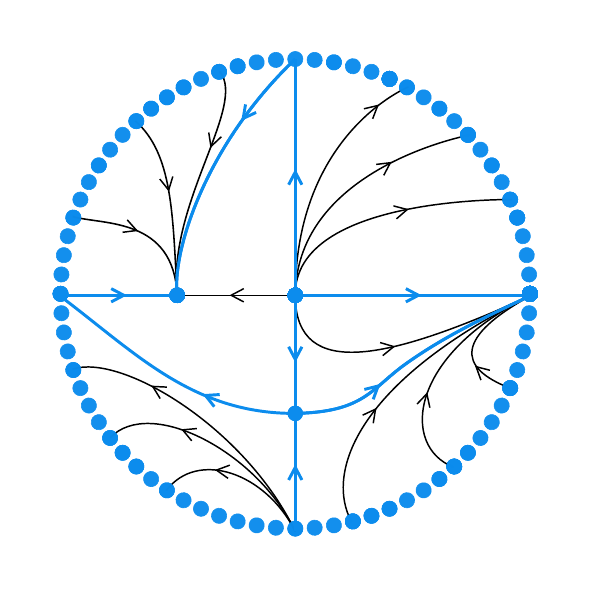}}{(R3)}
	\stackunder[2pt]{\includegraphics[width=3cm]{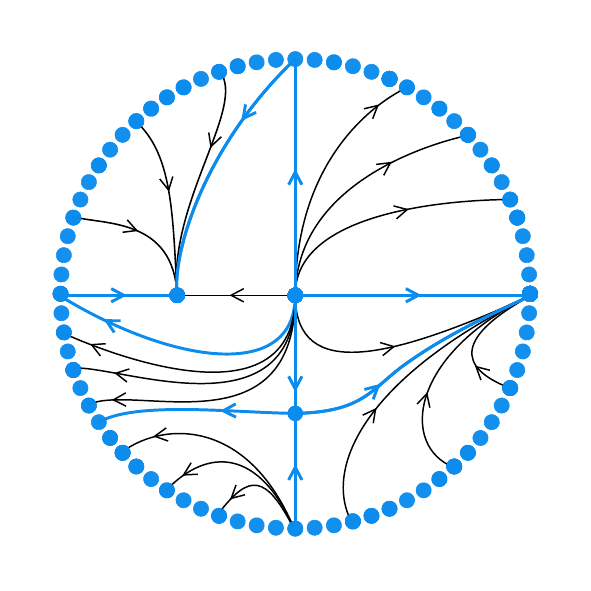}}{(R4)}
	\stackunder[2pt]{\includegraphics[width=3cm]{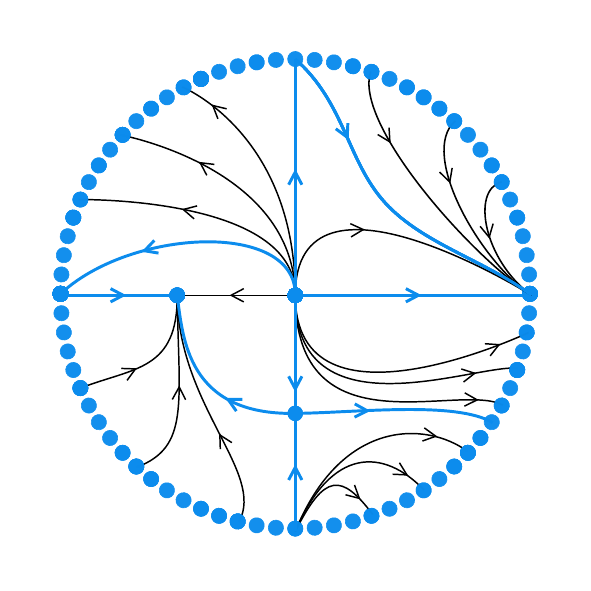}}{(R5)}
	\stackunder[2pt]{\includegraphics[width=3cm]{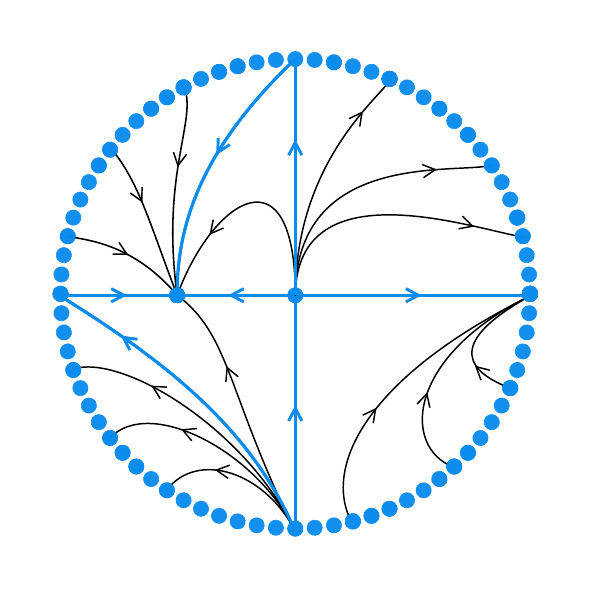}}{(R6)}
	\stackunder[2pt]{\includegraphics[width=3cm]{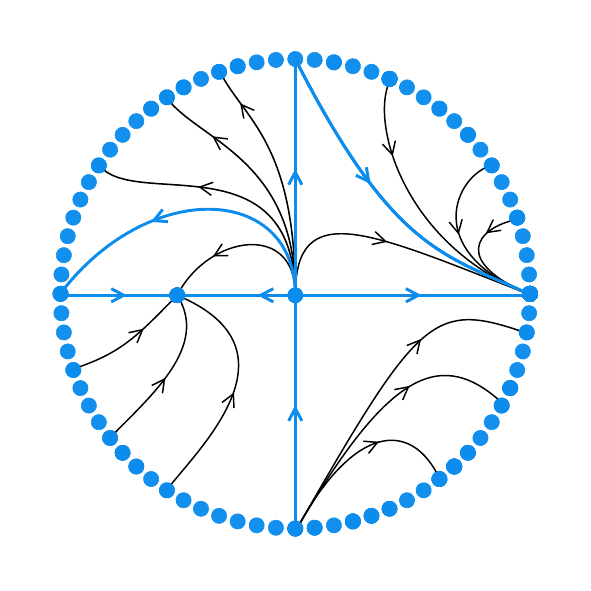}}{(R7)}
	\stackunder[2pt]{\includegraphics[width=3cm]{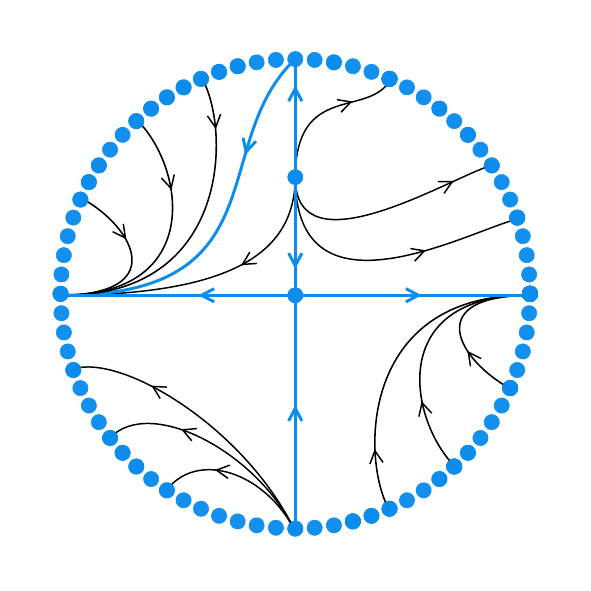}}{(R8)}
	\stackunder[2pt]{\includegraphics[width=3cm]{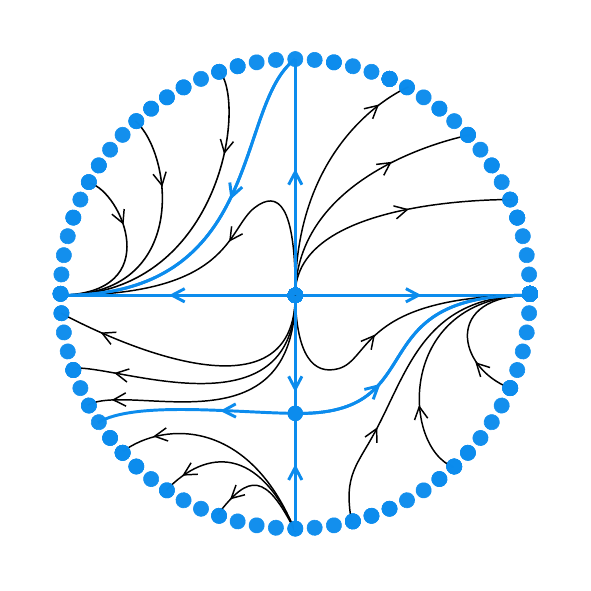}}{(R9)}
	\stackunder[2pt]{\includegraphics[width=3cm]{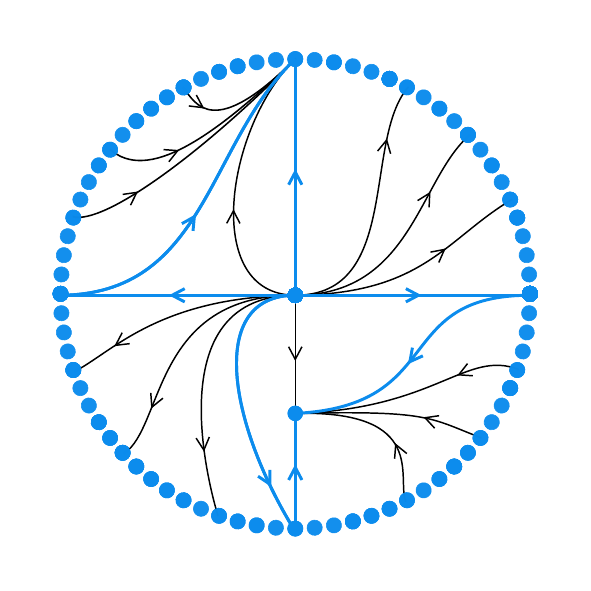}}{(R10)}
	\stackunder[2pt]{\includegraphics[width=3cm]{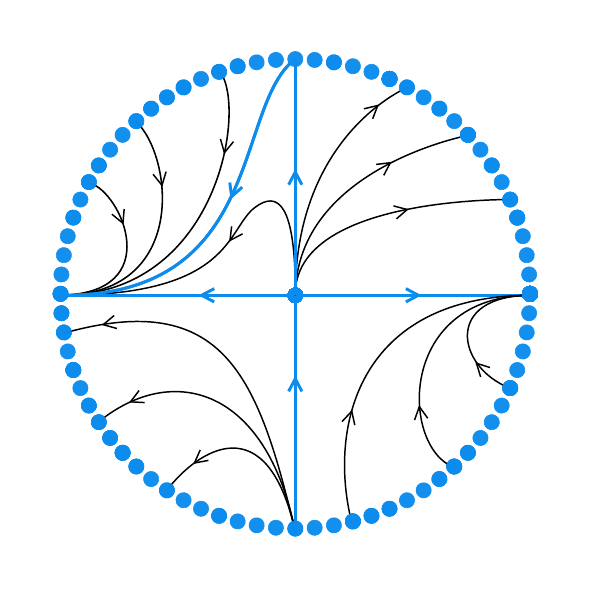}}{(R11)}
	\stackunder[2pt]{\includegraphics[width=3cm]{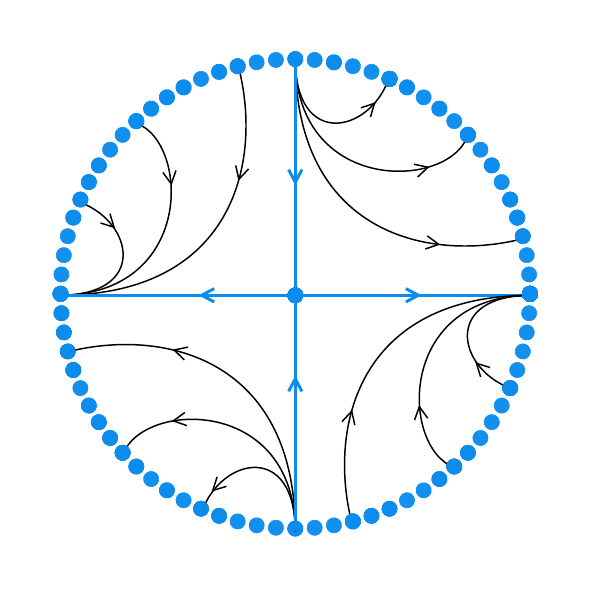}}{(R12)}
	\stackunder[2pt]{\includegraphics[width=3cm]{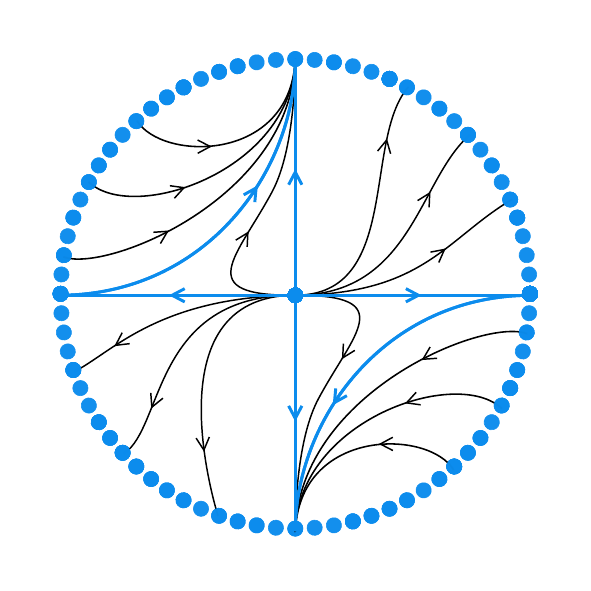}}{(R13)}
	\caption{The topologically distinct phase portraits of systems \eqref{sis2mu} in the Poincar\'e disc. The orbits drawn in blue color are the separatrices, while the drawn orbits which are not separatrices are drawn in black color}
	\label{fig:global_top}
\end{figure}

In order to give a detailed proof of Theorem \ref{th_global}, in Section \ref{sec:preliminaries} we give some definitions and results that will be useful.
In Sections \ref{sec:finite} and  \ref{sec:infinite} we give, respectively, the classification of the local phase portraits of the finite and infinite singular points, and in Section \ref{sec:global} we study the global phase portraits on the Poincaré disc to prove Theorem \ref{th_global}.

\section{Preliminaries}\label{sec:preliminaries}

We shall study systems \eqref{sis2mu} and it will suffice to do so under the conditions given in the following result.

\begin{proposition}
	To determine all global phase portraits of systems \eqref{sis2mu} it is sufficient to study those systems whose parameters satisfy the following conditions:
	\begin{align*}
		H= \left\lbrace b_1\neq 0, c_0 - b_0 \neq 0, b_0\geq 0, b_2\geq 0, b_3\geq 0, b_3^2+c_0^2\neq 0, b_2^2+b_0^2\neq0  \right\rbrace.
	\end{align*}
	If $b_2b_3=0$ then it is enough to study the case with $b_1>0$, and if $b_0=0$ it is enough to consider $c_0>0$. 
\end{proposition}

\begin{proof}
	The proof follows from Propositions 1 and 2, and Corollary 4.1 of \cite{SIS2}. In those results it is proved that Kolmogorov systems \eqref{SIS2} can be reduced to satisfy conditions
	\begin{align*}
		\tilde{H}= \left\lbrace b_1\neq 0, b_0 \mu+c_0 \neq 0, b_0\geq 0, b_2\geq 0, b_3\geq 0, (\mu b_3)^2+c_0^2\neq 0, b_2^2+b_0^2\neq0  \right\rbrace,
	\end{align*}
	either using symmetries, or eliminating known phase portraits, or eliminating phase portraits with infinitely many finite singular points. Asumming $\mu=-1$ the conditions $\tilde{H}$ become the conditions $H$ given above. Also it was proved in the mentioned results of \cite{SIS2} that if $b_2b_3=0$ then it is enough to study the case with $b_1>0$, as case $b_1<0$ can be reduced to this one by symmetry. Similarly, in the case with $b_0=0$ it is enough to consider $c_0>0$. 
\end{proof}

As we want to study the global dynamics of systems \eqref{sis2mu}
we must determine the behaviour of the orbits near the infinity. In order to do that we will use the \textit{Poincar\'e compactification}.

We call $\mathbb{S}^2=\left\lbrace  y \in \mathbb{R}^3 : y_1^2+y_2^2+y_3^2=1 \right\rbrace$ the Poincaré sphere, and we will consider a polynomial vector field defined in its tangent plane at the point $(0,0,1)$.
Let that field be $X=(P(x,y),Q(x,y))$ and $d$ the maximum of the degrees of the polynomials $P$ and $Q$.  We can obtain another vector field $\overline{X}$ on $\mathbb{S}^2\textbackslash \mathbb{S}^1$ by means of the differentials $Df^+$ and $Df^-$ of the central projections. Although $\overline{X}$ is not defined on the equator $\mathbb{S}^1$, which corresponds with the points of the infinity of $\mathbb{R}^2$, multiplying by $y_3^d$ we can extended it analytically to another vector field $\rho(X)$ defined on the closed Poincaré sphere. We say that $\rho(X)$ is the \textit{Poincar\'e compactification} of the vector field $X$ on $\mathbb{R}^2$.

We will work with the expressions of the Poincaré compactification in the local charts $(U_i,\phi_i)$ and $(V_i,\psi_i)$ where,  for $i=1,2,3$:
\begin{equation*}
	U_i= \left\lbrace w \in \mathbb{S}^2 :w_i > 0  \right\rbrace, \;\;\;\, \phi_i: U_i \longrightarrow \mathbb{R}^2, \;\;\;\;  V_i = \left\lbrace w \in \mathbb{S}^2 : w_i < 0  \right\rbrace, \;\;\;\; \psi_i: V_i \longrightarrow \mathbb{R}^2,
\end{equation*}
and $\phi_i(w) = \psi_i(w) = \left( w_m/w_i, w_n/w_i \right)$ for $m < n$ and $m,n \neq i$. 

As the field $\rho(X)$ on $\mathbb{S}^2$ is symmetric with respect to the origin of $\mathbb{R}^3$, it will be enough to study the orbits on the closed northern hemisphere of $\mathbb{S}^2$, which we will project  onto the plane $y_3=0$ by means of the orthogonal projection, so we will draw the global phase portraits in the so called 
\textit{Poincar\'e disc}, denoted by $\mathbb{D}^2$.

We can cover all the Poincaré disc with the charts $U_1$, $U_2$, $V_1$ and $V_2$ so it will be no necessary to study the expressions of the field on $U_3$ and $V_3$. Also, as the expression for $\rho(X)$ in the local charts $(V_i,\psi_i)$, with $i=1,2$, can be obtained  multiplying by $(-1)^{d-1}$ the expression in $(U_i,\phi_i)$, it will be enough to study the Poincaré compactification in $U_1$ and $U_2$. 

The expression of $\rho(X)$ in the local chart $(U_1,\phi_1)$ is
\begin{equation}\label{Poincare_comp_U1}
	\dot{u}= v^d \left[ -u \: P\left( \frac{1}{v}, \frac{u}{v}\right)  + Q\left( \frac{1}{v}, \frac{u}{v}\right) \right], \;
	\dot{v}= - v^{d+1} \: P\left( \frac{1}{v}, \frac{u}{v}\right),
\end{equation}
and in the local chart $(U_2,\phi_2)$ is
\begin{equation}\label{Poincare_comp_U2}
	\dot{u}= v^d \left[ P\left( \frac{1}{v}, \frac{u}{v}\right) - u Q\left( \frac{1}{v}, \frac{u}{v}\right) \right], \;
	\dot{v}= - v^{d+1} \: P\left( \frac{1}{v}, \frac{u}{v}\right).
\end{equation}

The Poincaré compactification will allow to study the \textit{infinite singular points} of $X$, which are the singular points of $\rho(X)$  over boundary of the Poincaré disc. Note that if we have a singular point $p\in \mathbb{S}^1$ then the opposite $-p$ is also a singular point and it has the same stability if $d$ is odd and opposite  stability if $d$ is even. We note that the points at the infinity in the local charts $U_i$ and $V_i$ for $i=1,2$ have coordinates $(u,0)$. For more details about the Poincaré compactification see Chapter 5 of \cite{Libro}.


To draw and classify the phase portraits on Poincaré disc, we have to pay special attention to the \textit{separatrices}. The separatrices of a polynomial differential system are:
\begin{enumerate}
	\item[(i)] the singular points;
	\item[(ii)] the periodic orbits for which there does not exist a neighborhood entirely consisting of periodic orbits; 
	\item[(iii)] all the orbits at the infinity;
	\item[(iv)] the orbits $\gamma(p)$ homeomorphic to $\mathbb{R}$ for which there does not exist a neighborhood $N$ of $\gamma(p)$ such that
	\begin{enumerate}
		\item[(1)] for all $q\in N$, $\alpha(q)=\alpha(p)$ and $\omega(q)=\omega(p)$,
		\item[(2)] the boundary $\partial N$ of $N$, that is $\partial N= \overline{N} \textbackslash N$, is formed by $\alpha(p)$, $\omega(p)$ and two orbits $\gamma(q_1)$ and $\gamma(q_2)$ such that $\alpha(p)=\alpha(q_1)=\alpha(q_2)$ and $\omega(p)=\omega(q_1)=\omega(q_2)$. 
	\end{enumerate}
\end{enumerate}
In an analytic system with isolated singular points this corresponds with 
all the orbits at the infinity, the finite singular points, the orbits on the boundary of a hyperbolic sector, and the limit cycles, for more details see \cite{Neumann}.

We call \textit{canonical regions} to each one of the
connected components resulting from removing all the sepa\-ratrices from $\mathbb{D}^2$, and \textit{separatrix configuration} of  $\pi(\rho(X))$ to the union of all the separatrices together with an orbit of each canonical region. 

We recall that two polynomial vector fields $X_1$ and $X_2$ are \textit{topologically equivalent} if there exists a homeomorphism on the Poincar\'e disc that sends orbits of $X_1$ to orbits of $X_2$, preserving or reversing the orientation of all the orbits, and it also preserves the infinity. The same definition is applicable to separatrix configurations.

The following result of Markus \cite{Markus}, Neumann \cite{Neumann} and Peixoto \cite{Peixoto} allows to study only the separatrix configurations to determine the topological classification of a polynomial differential system in the Poincaré disc.

\begin{theorem}\label{th_MNP}
	The phase portraits in the Poincaré disc of two compactified polynomial vector fields $\pi(\rho(X_1))$ and $\pi(\rho(X_2))$ with finitely many separatrices are topologically equivalent if and only if their separatrix configurations are topologically equivalent.
\end{theorem}

Although this result can be applied only to vector fields with finitely many separatrices, and this will not be the case of systems \eqref{sis2mu}, we can apply it to those systems in the open Poincaré disc. If two phase portraits are topologically distinct in the open Poincaré disc, they will be distinct in the closed disc and if two phase portraits are topologically equivalent in the open Poincaré disc, they will be still equivalent if we add the boundary filled of singular points and consider the closed Poincaré disc.

As we have already mentioned systems \eqref{sis2mu} will have an infinite number of singular points, namely all points at infinity. To study those singular points which form a continuum we will need the following result, which can be found in \cite{Devaney,Hirsch}. 
Let  $\varphi_t$ be a smooth flow on a manifold $M$, and consider a submanifold $C$ consisting entirely of singular points. The submanifold $C$ is said \textit{normally hyperbolic} if the tangent bundle to $M$ over $C$ splits into three subbundles $TC$, $E^s$ and $E^u$ invariant under the flow and satisfying that
$d\varphi_t$ contracts (respectively, expands) $E^s$ (respectively, $E^u$) exponentially and $TC$ is the tangent bundle of $C$. 

\begin{theorem}\label{th_normhyp}
	Let $C$ be a normally hyperbolic submanifold consisting of singular points for a flow $\varphi_t$. Then there exist smooth stable and unstable manifolds tangent along $C$ to $E^s \oplus TC$ and $E^u\oplus TC$ respectively. Furthermore, both $C$ and the stable and unstable manifolds are permanent under small perturbations of the flow.
\end{theorem}

\section{Local phase portraits of the finite singular points}\label{sec:finite}

Asumming the condition $\mu=-1$, from Section 5 in \cite{SIS2} we know that the singular points of systems \eqref{sis2mu} are
\begin{equation*}
	P_0=(0,0),
	\hspace{0.4cm}		
	P_1=\left(0, -\dfrac{c_0}{ b_3} \right) \;\text{if} \; b_3 \neq0 	\hspace{0.2cm}  \text{and} \hspace{0.2cm}
	P_2=\left( -\dfrac{b_0}{b_2},0\right)\; \text{if} \;  b_2\neq0, 
\end{equation*}	
and from Table 1 of \cite{SIS2}, we distinguish four cases depending on the existence of the singular points. These cases are given in Table \ref{tab:cases1-4}.

\begin{table}[H]
	\begin{center}
	\begin{tabular}{|c|ll|}
			\hline
			\textbf{Case} & \textbf{Conditions} & \textbf{Finite singular points} \\
			\hline
			\hline
			1&  $ b_3 \neq 0$, $b_2\neq 0$. & $P_0$, $P_1$, $P_2$. \\	\hline
			2 &  $ b_3 \neq 0$, $b_2=0$, $b_0\neq0$.  & $P_0$, $P_1$.\\
			\hline
			3 &  $ b_3=0$, $c_0\neq 0$, $b_2\neq 0$.  & $P_0$, $P_2$. \\
			\hline
			4 &  $ b_3=0$, $c_0\neq 0$, $b_2=0$, $b_0\neq 0$.  & $P_0$.  \\
			\hline
		\end{tabular}
			\caption{The different cases for the finite singular points.}
		\label{tab:cases1-4}
			\end{center}
\end{table}

Also from \cite{SIS2}, taking $\mu=-1$ in Lemma 1 and Tables 2 to 5,
we get the following local classification in 15 subcases for the finite singular points.

\begin{table}[H]
	\begin{center}
	\begin{tabular}{|c|p{4.7cm}p{7.8cm}|}
			\multicolumn{3}{l}{ \textbf{Case 1: } $\: \boldsymbol{ b_3 \neq 0}$, $\boldsymbol{b_2 \neq 0}$.} \smallskip\\
			\hline
			\hline
			\textbf{Sub.}& \textbf{Conditions} & \textbf{Classification} \\
			\hline
			\hline
			
			1.1
			&
			$b_0>0$, $c_0<0$.
			&
			$P_0$ saddle, $P_1$ unstable node,  $P_2$ stable node.\\
			\hline
			
			1.2
			&
			$b_0>0$, $c_0>0$, $c_0-b_0<0$.
			&
			$P_0$ unstable node, $P_1$ saddle, $P_2$ stable node.\\
			\hline

			1.3
			&
			$b_0>0$, $c_0>0$, $c_0-b_0>0$.
			&
			$P_0$ unstable node, $P_1$ stable node,  $P_2$ saddle.\\
			\hline

			1.4
			&
			$c_0=0$, $b_0>0$.
			&
			$P_0\equiv P_1$ saddle-node, $P_2$ stable node.\\
			\hline

			1.5
			&
			$b_0=0$, $c_0>0$.
			&
			$P_0\equiv P_2$ saddle-node,  $P_1$ stable node.\\
			\hline
			
		\end{tabular}
			\caption{Classification of the local phase portraits of the finite singular points of case 1 of Table \ref{tab:cases1-4}.}
		\label{tab:cases_fin_local_1}
		\end{center}
\end{table}	

\begin{table}[H]
	\begin{center}
	\begin{tabular}{|c|p{5.5cm}p{6.5cm}|}
			\multicolumn{3}{l}{ \textbf{Case 2: } $\: \boldsymbol{b_3\neq 0}$, $\boldsymbol{b_2=0}$, $\boldsymbol{b_0\neq 0}$.} \smallskip\\
			\hline
			\hline
			\textbf{Sub.}& \textbf{Conditions} & \textbf{Classification} \\
			\hline
			\hline
			
			2.1
			&
			$b_0>0$, $c_0<0$,  $c_0-b_0<0$.
			&
			$P_0$ saddle, $P_1$ unstable node.  \\
			\hline
			
			2.2
			&
			$b_0>0$, $c_0>0$,  $c_0-b_0<0$.
			&
			$P_0$ unstable node,  $P_1$ saddle.\\
			\hline
			
			2.3
			&
			$b_0>0$, $c_0>0$,  $c_0-b_0>0$.
			&
			$P_0$ unstable node, $P_1$ stable node.\\
			\hline
			
			2.4
			&
			$c_0=0$, $b_0>0$.
			&
			$P_0\equiv P_1$ saddle-node.\\
			\hline
			
		\end{tabular}
		\label{tab:cases_fin_local_2}
			\caption{Classification of the local phase portraits of the finite singular points of case 2 of Table \ref{tab:cases1-4}.}
			\end{center}
\end{table}

\begin{table}[H]
	\begin{center}
	\begin{tabular}{|c|p{5.5cm}p{6.5cm}|}
			\multicolumn{3}{l}{ \textbf{Case 3: } $\: \boldsymbol{b_3=0}$, $\boldsymbol{c_0\neq0}$, $\boldsymbol{b_2\neq 0}$.} \smallskip\\
			\hline
			\hline
			\textbf{Sub.}& \textbf{Conditions} & \textbf{Classification} \\
			\hline
			\hline

			3.1
			&
			$b_0>0$, $c_0<0$,  $c_0-b_0<0$.
			&
			$P_0$ saddle, $P_2$ stable node.\\
			\hline
			
			3.2
			&
			$b_0>0$, $c_0>0$, $c_0-b_0>0$.
			&
			$P_0$ unstable node, $P_2$ saddle. \\
			\hline
			
			3.3
			&
			$b_0>0$, $c_0>0$, $c_0-b_0<0$.
			&
			$P_0$ unstable node, $P_2$ stable node. \\
			\hline
			
			3.4
			&
			$b_0=0$, $c_0>0$.
			&
			$P_0\equiv P_2$ saddle-node. \\
			\hline
			
		\end{tabular}
			\caption{Classification of the local phase portraits of the finite singular points of case 3 of Table \ref{tab:cases1-4}.}
		\label{tab:cases_fin_local_3}
		\end{center}
\end{table}

\begin{table}[H]
	\begin{center}
	\begin{tabular}{|c|p{5.5cm}p{6.5cm}|}
			\multicolumn{3}{l}{ \textbf{Case 4: } $\: \boldsymbol{b_3=0}$, $\boldsymbol{c_0\neq 0}$, $\boldsymbol{b_2=0}$, $\boldsymbol{b_0\neq0}$.} \smallskip\\
			\hline
			\hline
			\textbf{Sub.}&\textbf{Conditions} & \textbf{Classification} \\
			\hline
			\hline
			
			4.1
			&
			$b_0>0$, $c_0<0$.
			&
			$P_0$ saddle.\\
			\hline
			
			4.2
			&
			$b_0>0$, $c_0>0$.
			&
			$P_0$ unstable node.\\
			\hline
			
		\end{tabular}
			\caption{Classification of the local phase portraits of the finite singular points of case 4 of Table \ref{tab:cases1-4}.}
		\label{tab:cases_fin_local_4}
		\end{center}
\end{table}

\section{Local phase portraits at the infinite singular points}\label{sec:infinite}

Here we study the local phase portrait at the infinite singular points, and as it was said previously, we work under the hypothesis $H$. 
The expression of the Poincaré compactification of systems \eqref{sis2mu} in the local chart $U_1$ according to equations \eqref{Poincare_comp_U1} is
\begin{equation}\label{system2U1}
	\begin{split}
		\dot{u}&= (c_0-b_0)uv^2,\\
		\dot{v}&= -b_3 u v^2 - b_0 v^3 - b_1 u v - b_2 v^2.
	\end{split}
\end{equation}
In the chart $U_2$ according to equations \eqref{Poincare_comp_U2} the expression is
\begin{equation}\label{system2U2}
	\begin{split}
		\dot{u}&= (b_0-c_0)uv^2,\\
		\dot{v}&= -b_2 u v^2 - c_0 v^3 - b_1  u v + b_3 v^2.
	\end{split}
\end{equation}
We want to study all the points at the infinity, which correspond with the line $v=0$ of these systems. To do that it is enough to study the singular points over $v=0$ in the chart $U_1$ and the origin of the chart $U_2$. 

We easily check in system \eqref{system2U1} that all points over the line $v=0$ are singular points. The eigenvalues of the Jacobian matrix at these singular points are both zero at the origin and at any other point $(u_0,0)$ the eigenvalues are  zero and $-b_1u_0$. If $b_1>0$  (respectively, $b_1<0$),  the nonzero eigenvalue is positive (respectively, negative) at the points on the negative $u$-axis, which correspond with the infinite singular points at the second and fourth quadrants of the Poincaré disc; 
the nonzero eigenvalue is negative for the infinite points at the first and third quadrants on the Poincaré disc (respectively, positive). 
For the singular points on the continuum of singular points at infinity that have a non-zero eigenvalue their local phase portrait is determined by Theorem \ref{th_normhyp}, while for the singular points with two eigenvalues being zero, their phase portrait is studied by eliminating the common factor of the systems \eqref{system2U1}, so that the singular point with the two zero eigenvalues is now an isolated singular point of the new differential systems. Then, by Theorem \ref{th_normhyp} we get the following result:

\begin{lemma}\label{lemma_inf}
	For all the infinite singular point of systems \eqref{sis2mu} distinct from the origin of the charts $U_1$ and $U_2$ the following statements hold.
	\begin{enumerate}
		\item[$\bullet$] If $b_1>0$, to  the points on the first and third quadrants arrives exactly one orbit from outside the infinity, and from the points on the second and fourth quadrants leaves exactly one orbit outside the infinity.
		\item[$\bullet$] If $b_1<0$,  from the points on the first and third quadrants leaves exactly one orbit outside the infinity, and to the points on the second and fourth quadrants arrives exactly one orbit from outside the infinity.
	\end{enumerate}
	
\end{lemma}

To study the origin of systems \eqref{system2U1} we eliminate a common factor $v$ from these systems and then study the singular points over the line $v=0$. We do that on Subsection \ref{subsec:O1} and there we prove Theorem \ref{lemma_O1}. The same occurs with the origin of the chart $U_2$, as the origin of systems \eqref{system2U2} is a singular point and the eigenvalues of the Jacobian matrix at that point are both zero. We study this point in Subsection \ref{subsec:O2} proving Theorem \ref{lemma_O2}. Note that Theorem \eqref{lemma_O1} and Theorem \eqref{lemma_O2} determine the local phase portrait at the origin of the charts $U_1$ and $U_2$ in the Poincaré disc, but also at the origins of charts $V_1$ and $V_2$.


\begin{theorem}\label{lemma_O1}
	The origin of systems \eqref{system2U1} is a singular point and it has $3$ topologically distinct local phase portraits, which taking into account the position of the sectors and orientation of the orbits give raise to the $8$ phase portraits described in Figure $\ref{fig:localppO1}.$
\end{theorem}

\begin{figure}[H] 
	\centering
	\scriptsize
	\stackunder[2pt]{\includegraphics[width=2.4cm]{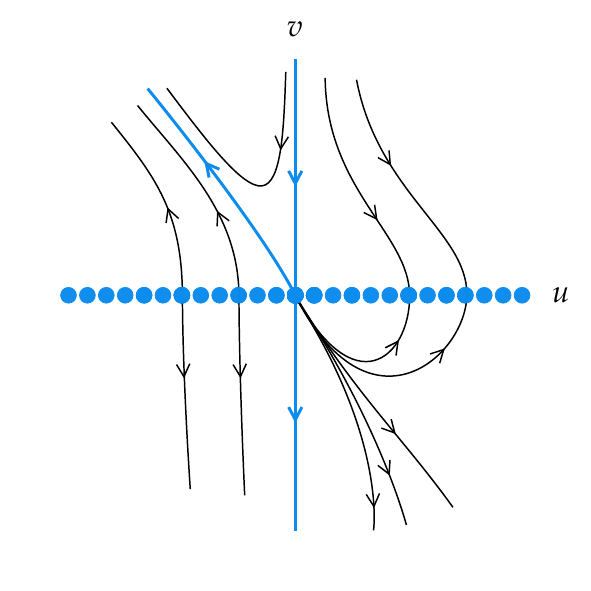}}{$L^1_{1}$}
	\stackunder[2pt]{\includegraphics[width=2.4cm]{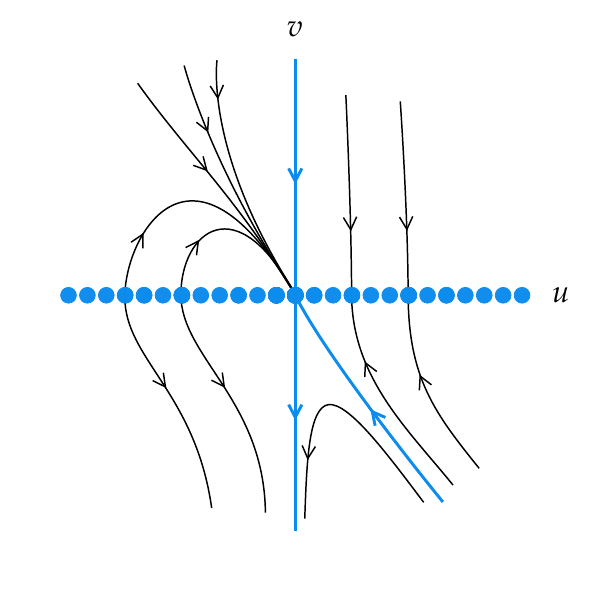}}{$L^1_{2}$}
	\stackunder[2pt]{\includegraphics[width=2.4cm]{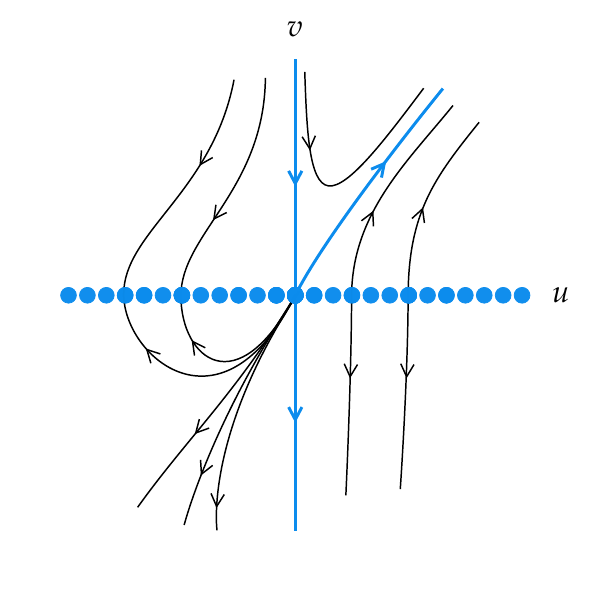}}{$L^1_{3}$}
	\stackunder[2pt]{\includegraphics[width=2.4cm]{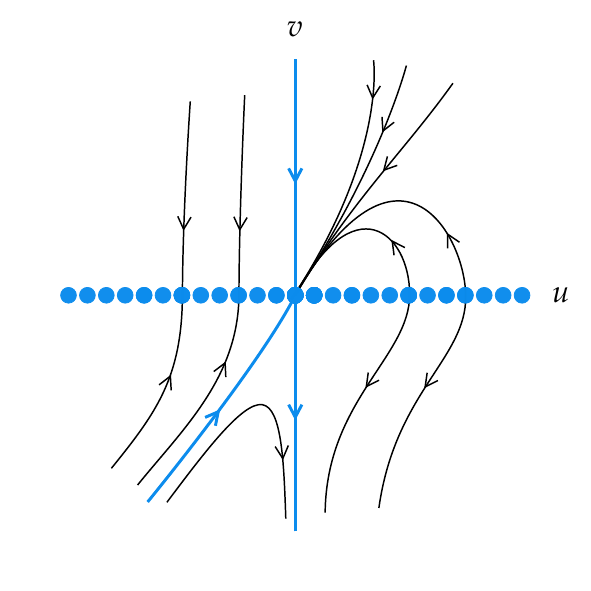}}{$L^1_{4}$}
	
	\vspace{0.3cm}
	
	\stackunder[2pt]{\includegraphics[width=2.4cm]{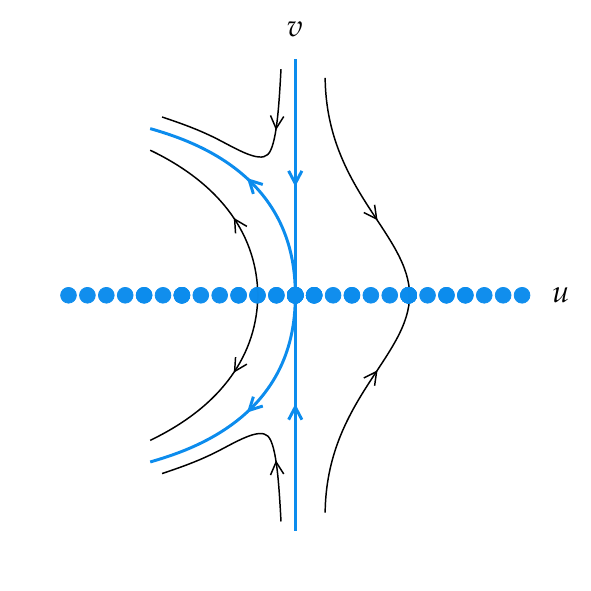}}{$L^1_{5}$}
	\stackunder[2pt]{\includegraphics[width=2.4cm]{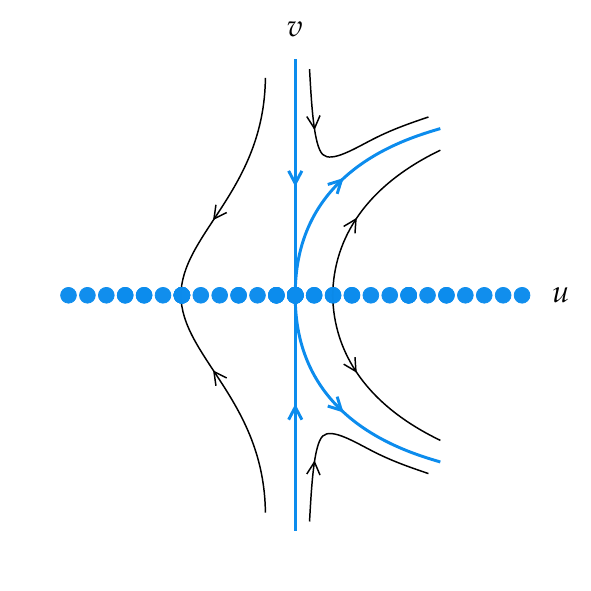}}{$L^1_{6}$}
	\stackunder[2pt]{\includegraphics[width=2.4cm]{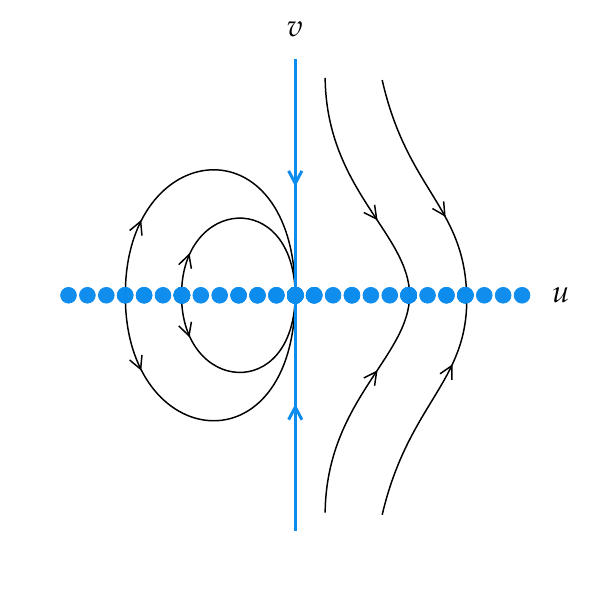}}{$L^1_{7}$}
	\stackunder[2pt]{\includegraphics[width=2.4cm]{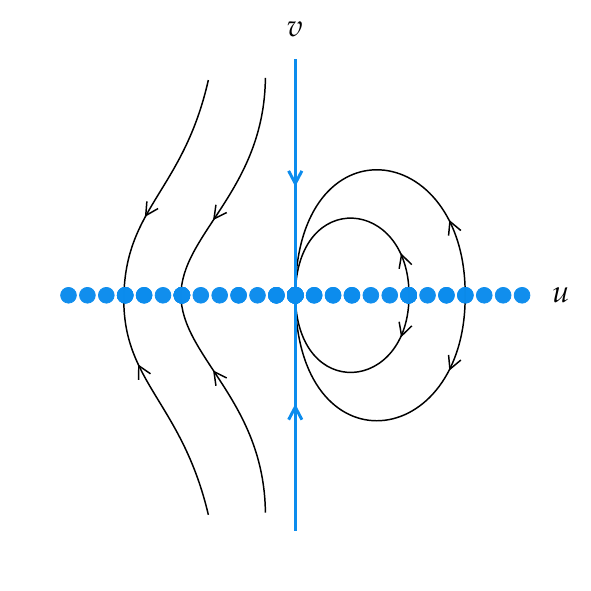}}{$L^1_{8}$}
	\caption{Local phase portraits at the infinite singular point $O_1$.}
	\label{fig:localppO1}
\end{figure}

\begin{remark}
	Note that phase portraits $L^1_1$ to $L^1_4$ correspond to the first equivalence class, $L^1_5$ and $L^1_6$ to the second class, and $L^1_7$ and $L^1_8$ to the third class.
\end{remark}

\begin{theorem}\label{lemma_O2}
	The origin of the chart $U_2$ is an infinite singular point of systems \eqref{sis2mu} and it has $3$ topologically distinct local phase portraits, which taking into account the position of the sectors and orientation of the orbits give raise to the $10$ phase portraits described in Figure $\ref{fig:localppO2}$.
\end{theorem}

\begin{figure}[H] 
	\centering
	\scriptsize
	\stackunder[2pt]{\includegraphics[width=2.4cm]{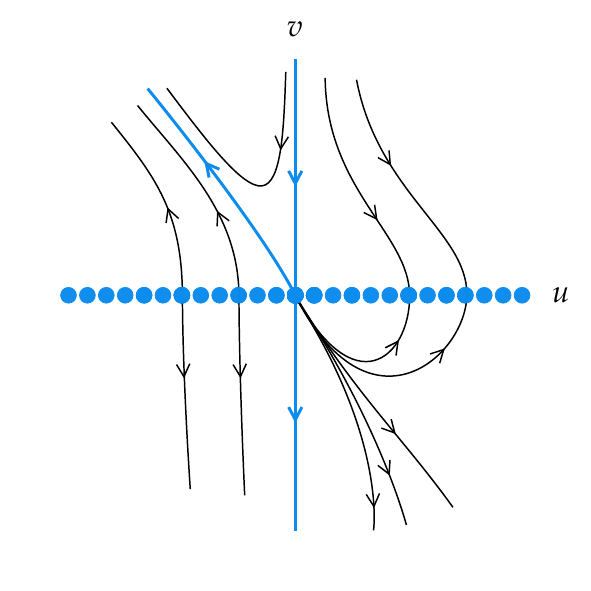}}{$L^2_{1}$}
	\stackunder[2pt]{\includegraphics[width=2.4cm]{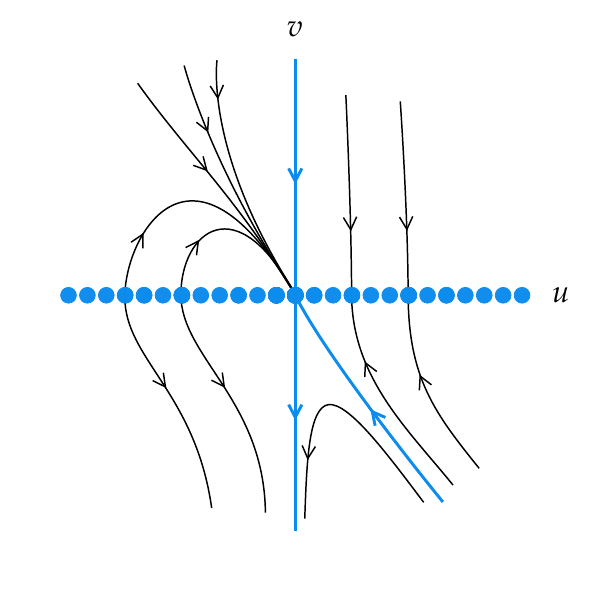}}{$L^2_{2}$}
	\stackunder[2pt]{\includegraphics[width=2.4cm]{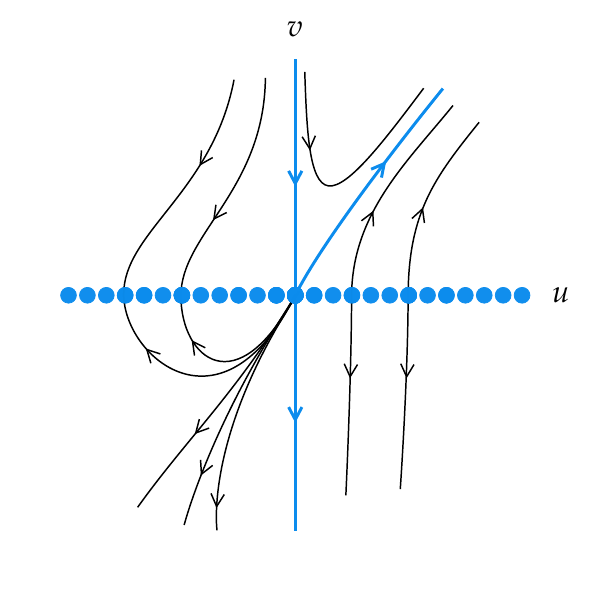}}{$L^2_{3}$}
	\stackunder[2pt]{\includegraphics[width=2.4cm]{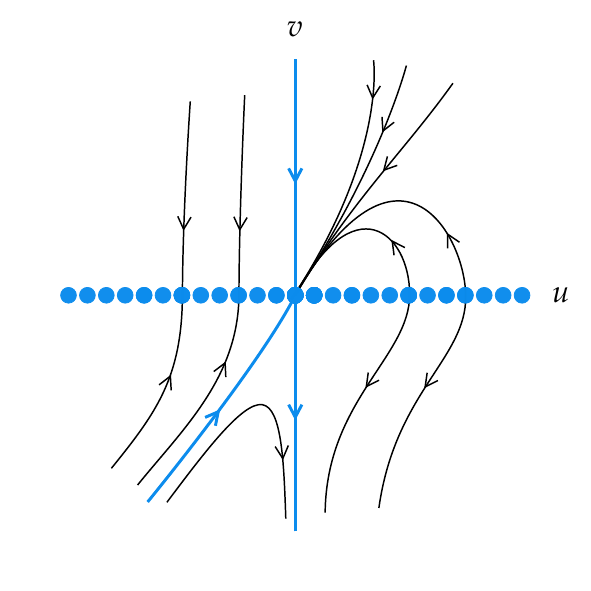}}{$L^2_{4}$}
	
	\vspace{0.3cm}
	
	\stackunder[2pt]{\includegraphics[width=2.4cm]{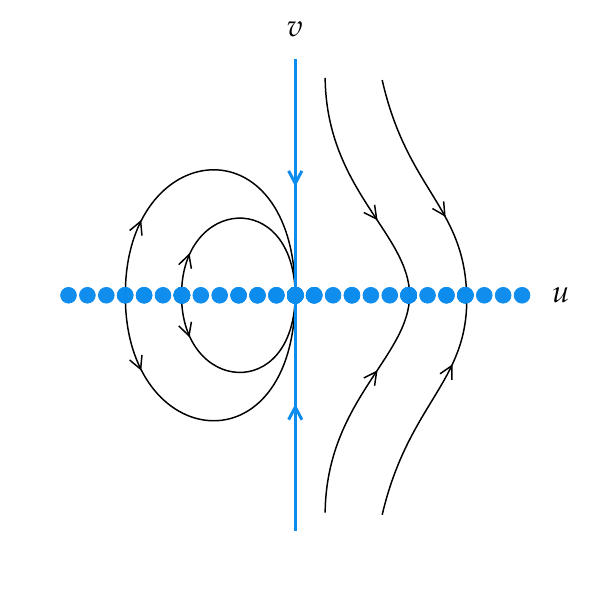}}{$L^2_{5}$}
	\stackunder[2pt]{\includegraphics[width=2.4cm]{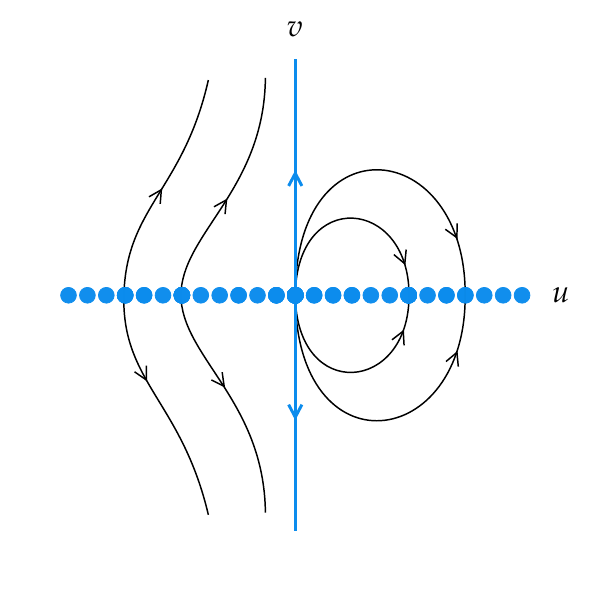}}{$L^2_{6}$}
	\stackunder[2pt]{\includegraphics[width=2.4cm]{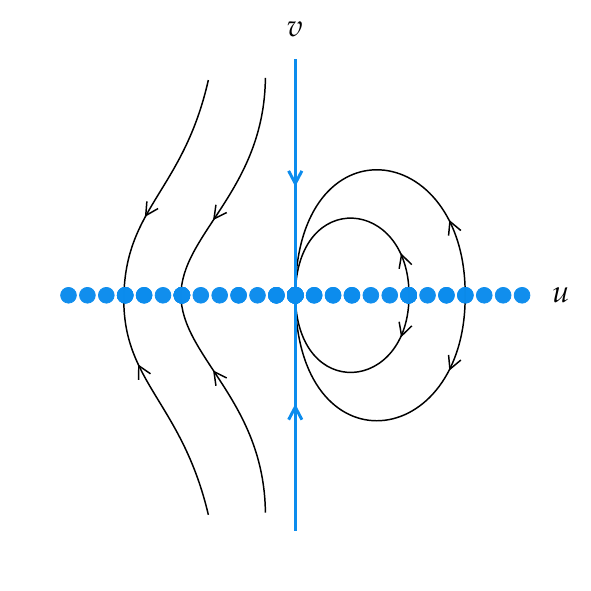}}{$L^2_{7}$}
	\stackunder[2pt]{\includegraphics[width=2.4cm]{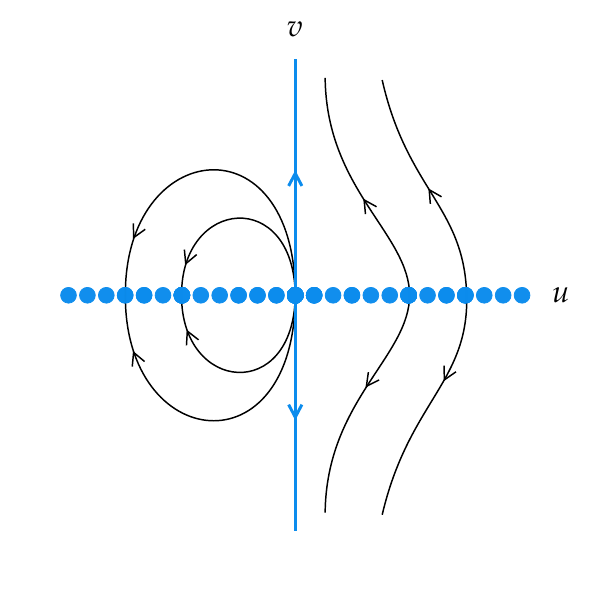}}{$L^2_{8}$}
	\stackunder[2pt]{\includegraphics[width=2.4cm]{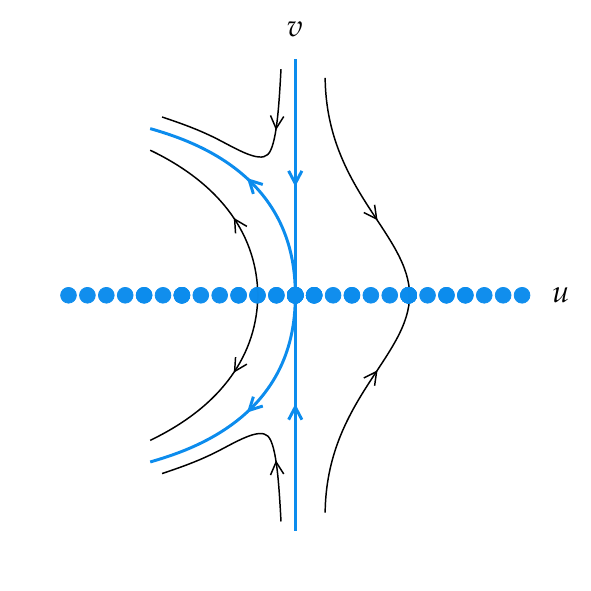}}{$L^2_{9}$}
	\stackunder[2pt]{\includegraphics[width=2.4cm]{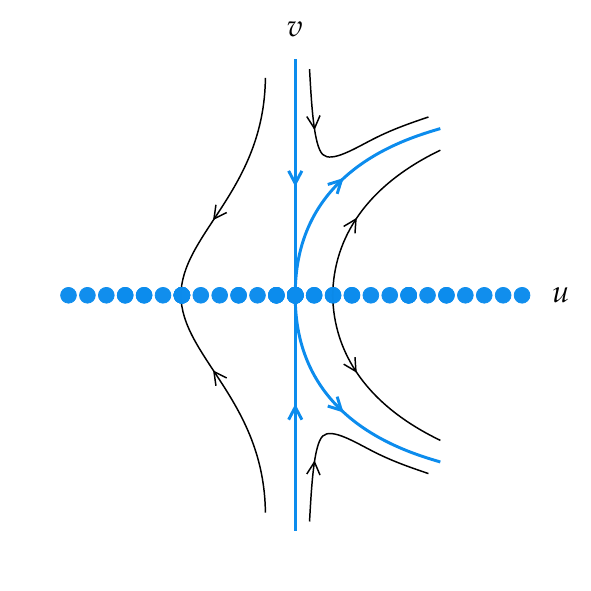}}{$L^2_{10}$}
	\caption{Local phase portraits at the infinite singular point $O_2$.}
	\label{fig:localppO2}
\end{figure}

\begin{remark}
	Note that phase portraits $L^1_1$ to $L^1_4$ correspond to the first equivalence class, $L^1_5$ to $L^1_8$ to the second class, and $L^1_9$ and $L^1_{10}$ to the third class.
\end{remark}

\subsection{Study of the origin of the chart $U_1$}\label{subsec:O1}

To study the origin of the chart $U_1$, first
we eliminate a common factor $v$ from systems \eqref{system2U1} obtaining: 
\begin{equation}\label{sisU1-fc}
	\begin{split}
		\dot{u}&= (c_0-b_0) u v, \\
		\dot{v}&= -b_3 u v - b_0 v^2 -b_1 u - b_2 v.
	\end{split}
\end{equation}
The only singular point of these systems over $v=0$ is the origin, and the eigenvalues of the Jacobian matrix at that point are zero and $-b_2$. Then this singular point can be semi-hyperbolic or nilpotent.

\vspace{0.2cm}
\textbf{ Semi-hyperbolic case}. If $b_2\neq0$ then the origin of systems \eqref{sisU1-fc} is semi-hyperbolic so its phase portrait can be determined by Theorem 2.19 in \cite{Libro}, concluding that it is always a saddle-node. In order to determine its local phase portrait it will be necessary to know the position of the different sectors and the orientation of the orbits in the saddle-node, so we must determine these depending on the parameters. 

If $b_1>0$, $c_0-b_0>0$ and $b_0=0$, then by the information given by the theorem and the sense of the flow in the different regions, the position of the sectors of the saddle-node and the orientation of the orbits for systems \eqref{sisU1-fc} is the one given in Figure \ref{fig:O1semi1}(a). To obtain the local phase portrait of the origin of the chart $U_1$ we must multiply by $v$, so that all the points over the $v$-axis  become singular points and the orbits on the third and fourth quadrants reverse their orientation. Thus we obtain the phase portrait of Figure \ref{fig:O1semi1}(b), which is also $L^1_1$ of Figure \ref{fig:localppO1}.

\begin{figure}[H] 
	\centering
	\scriptsize
	\stackunder[2pt]{\includegraphics[width=3cm]{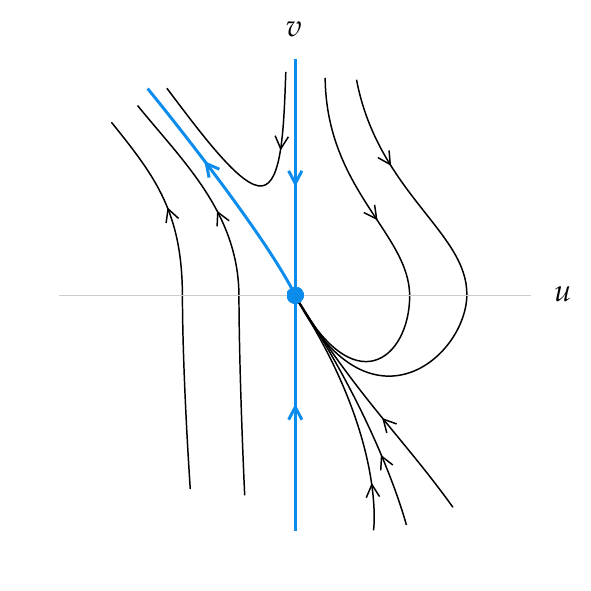}}{(a)}
	\stackunder[2pt]{\includegraphics[width=3cm]{RFO1/L_1_1}}{(b)}
	\caption{Local phase portraits of the origins of systems \eqref{sisU1-fc} and \eqref{system2U1} with $b_1>0$ $c_0-b_0>0$ and $b_0=0$.}
	\label{fig:O1semi1}
\end{figure}

If $b_1>0$, $c_0-b_0>0$ and $b_0>0$, the fact that the parameter $b_0$ is nonzero makes that systems \eqref{sisU1-fc} have a singular point on the negative $v$-axis, so that affects the phase portrait but not in a neighbourhood of the origin. Then we obtain the same phase portrait for $O_1$ as in the previous case, $L^1_1$.

Similarly we determine the position of the sectors and the orientation of the orbits in the remaining cases. If $b_1>0$, $c_0-b_0<0$ and $b_0>0$ we obtain the phase portrait $L^1_2$ of Figure \ref{fig:localppO1}. If $b_1<0$, $c_0-b_0>0$ and $b_0\geq0$ we obtain the phase portrait $L^1_3 $, and if $b_1<0$, $c_0-b_0<0$ and $b_0>0$ the phase portrait is $L^1_4$.

\vspace{0.2cm}
\textbf{Nilpotent case}.
If $b_2=0$ then the origin of systems \eqref{sisU1-fc} is nilpotent so its phase portrait can be determined by Theorem 3.5 in \cite{Libro}, which concludes that in this case the singular point is either a saddle or it has a local phase portrait consisting of a hyperbolic sector and an elliptic sector, depending on the parameters. It is also necessary to determine the position of the sectors and the orientation of the orbits, and in order to do that we must take into account the information given by the theorem and also analyze the sense of the flow in the different regions depending on the parameters. Once we have determined the local phase portrait for systems \eqref{sisU1-fc} we must multiply by $v$ so all the points over the line $v=0$ become singular points, and the orientation of the orbits on the third and fourth quadrants is reversed. Thus we obtain for $O_1$ the 8 phase portraits in Figure \ref{fig:localppO2} under the following conditions: 

If $b_2=0$, $b_0>0$, $b_1>0$, $c_0-b_0\neq2b_0b_1$ and $c_0-b_0>0$ the phase portrait at $O_1$ is $L^1_5$. We obtain the same phase portrait if $b_2=0$, $b_0>0$, $b_1>0$ and $c_0-b_0=2b_0b_1$.

If $b_2=0$, $b_0>0$, $b_1>0$, $c_0-b_0\neq2b_0b_1$ and $c_0-b_0<0$ the phase portrait at  $O_1$ is $L^1_7$. 

If $b_2=0$, $b_0>0$, $b_1<0$, $c_0-b_0\neq2b_0b_1$ and $c_0-b_0>0$ the phase portrait at $O_1$ is $L^1_6$. The same result is obtained for $b_2=0$, $b_0>0$, $b_1<0$ and $b_0-c_0=2b_0b_1$.

If $b_2=0$, $b_0>0$, $b_1<0$, $c_0-b_0\neq2b_0b_1$ and $c_0-b_0<0$   the phase portrait at $O_1$ is $L^1_8$.

\subsection{Study of the origin of the chart $U_2$}\label{subsec:O2}

As in the previous section to determine the phase portrait at the singular point $O_2$ we eliminate a common factor $v$ from systems \eqref{system2U2}. Then we study the singular points over the line $v=0$ of systems
\begin{equation}
	\begin{split}
		\dot{u}&=(b_0-c_0)u v,\\
		\dot{v}&=-b_2 u v - c_0 v^2 - b_1 u + b_3 v.
	\end{split}
\end{equation}
The only singular point is the origin, and it presents a similar behaviour than in the previous case: it is semi-hyperbolic if $b_3\neq0$ and it is nilpotent if $b_3=0$.  In the semi-hyperbolic case the singular point is always a saddle-node, and attending to the information given by Theorem 2.19 in \cite{Libro} and to the sense of the flow, we get four possibilities for the position and orientation of the sectors in the saddle-node, which are associated with their corresponding conditions in Table \ref{tab:condO2}. In the nilpotent case the singular point can be a saddle or have a hyperbolic and an elliptic sector. In the first case we found two possibilities for the position of the saddle, and in the second case we found four different cases attending to the position and orientation of the two sectors. The results are given in Table \ref{tab:condO2}.

\renewcommand{\arraystretch}{1.4}
\begin{table}[H]
	\begin{center}
\begin{tabular}{|l|c|}
			\hline
			\textbf{Conditions} & \textbf{ Phase portrait $O_2$} \\
			\hline
			\hline
			$ b_3 \neq 0$, $b_1>0$, $b_0-c_0>0$. & $L^2_1$  \\
			\hline
			$ b_3 \neq 0$, $b_1>0$, $b_0-c_0<0$, $c_0>0$. & $L^2_2$  \\	
			\hline
			$ b_3 \neq 0$, $b_1<0$, $b_0-c_0>0$. & $L^2_3$  \\	
			\hline
			$ b_3 \neq 0$, $b_1<0$, $b_0-c_0<0$, $c_0>0$. & $L^2_4$  \\	
			\hline
			$ b_3 \neq 0$, $c_0\neq0$, $b_1>0$, $b_0-c_0\neq 2b_1c_0$, $b_0-c_0<0$, $c_0>0$. & $L^2_5$  \\	
			\hline
			$ b_3 \neq 0$, $c_0\neq0$, $b_1>0$, $b_0-c_0\neq 2b_1c_0$, $b_0-c_0>0$, $c_0<0$.  & $L^2_6$  \\	
			\hline
			$ b_3 \neq 0$, $c_0\neq0$, $b_1>0$, $b_0-c_0\neq 2b_1c_0$, $b_0-c_0>0$, $c_0>0$.  &  	\multirow{2}{1cm}{\centering{$L^2_9$}}  \\	
			\cline{1-1}
			$ b_3 \neq 0$, $c_0\neq0$, $b_1>0$, $b_0-c_0=2b_1c_0$, $c_0>0$. &  	\multirow{2}{*}  {\phantom{a}}  \\	
			\hline
			$ b_3 \neq 0$, $c_0\neq0$, $b_1<0$, $b_0-c_0\neq 2b_1c_0$, $b_0-c_0<0$, $c_0>0$. &  	\multirow{2}{1cm}{\centering{$L^2_{10}$}}  \\	
			\cline{1-1}
			$ b_3 \neq 0$, $c_0\neq0$, $b_1<0$, $b_0-c_0=2b_1c_0$, $c_0>0$. &  	\multirow{2}{*}  {\phantom{a}}  \\	
			\hline
			$ b_3 \neq 0$, $c_0\neq0$, $b_1<0$, $b_0-c_0\neq 2b_1c_0$, $b_0-c_0>0$, $c_0>0$. & $L^2_7$  \\	
			\hline
			$ b_3 \neq 0$, $c_0\neq0$, $b_1<0$, $b_0-c_0\neq 2b_1c_0$, $b_0-c_0<0$, $c_0<0$. & $L^2_{8}$  \\	
			\hline
		\end{tabular}
			\caption{Conditions for each local phase portrait of $O_2$.}
		\label{tab:condO2}
		\end{center}
\end{table}

\section{Global phase portraits}\label{sec:global}

In this section we prove Theorem \ref{th_global} by obtaining all the possible global phase portraits from the local information obtained in Sections \ref{sec:finite} and \ref{sec:infinite}. In each case of Tables \ref{tab:cases_fin_local_1} to \ref{tab:cases_fin_local_4} we must consider two subcases by setting the sign of $b_1$, and once this sign is fixed the local phase portrait at the infinite singular points is determined by Lemma \ref{lemma_inf} and Theorems \ref{lemma_O1} and \ref{lemma_O2}. There is an exception to this which is case 4.2 in Table \ref{tab:cases_fin_local_4}, as in this case we must consider four subcases fixing also the sign of $c_0-b_0$. Thus we have 32 cases. 

According to Theorem \ref{th_MNP} we have to draw the separatrix configuration in each case. We recall that the separatrices are the finite and infinite singular points, the limit cycles and the separatrices of the hyperbolic sectors. Systems \eqref{sis2mu} do not have any limit cycles as if they had a limit cycle it must surround a finite singular point, but all the finite singular points are over invariant lines, particularly over the axes, so there are no limit cycles. Then we have to draw the local phase portraits of  the singular points and the separatrices of the hyperbolic sectors for which we have to determine their $\alpha$ and $\omega$-limits. In 30 of the 32 cases the place where the separatrices are born and die
is determined in a unique way, so we obtain the corresponding global phase portrait by drawing them and one orbit in each canonical region which does not have an infinite number of singular points in the boundary, and three orbits (representing the infinite number of them existing) in each canonical region with an infinite number of singular points in the boundary. 

The two remaining cases are 1.2 and 1.3 in Table \ref{tab:cases_fin_local_1}, with $b_1>0$. In these cases the $\alpha$ and $\omega$-limits are not determined in a unique way, and we can connect the sepatrices in three different ways. 

In case 1.2, if we fix $b_1>0$, we obtain the phase portraits G3, G4 and G5 of Figure \ref{fig:global}, depending on how we connect the separatrices on the third quadrant. We know from the local information that there is a separatrix which $\omega$-limit is the origin of the chart $V_1$ and a separatrix which $\alpha$-limit the saddle $P_1$ in the negative $z$-axis. Studying their possible $\alpha$ and $\omega$-limits, respectively, we obtain the configuration given in Figure \ref{fig:demsep}. Note that in the second case, the two separatrices are connected and so there is actually only one separatrix on the quadrant. 

If $b_2 = b_1 c_0 /b_3$ then the connection of the separatrices takes place on the invariant straight line $y=-c_0/b_3$ and we have the global phase portrait G4. The two remaining phase portraits can be obtained by perturbing just one parameter. For example setting the values $b_0=2$, $b_1=b_2=b_3=1$, we obtain phase portrait G3 for $c_0=1/2$, G4 for $c_0=1$ and G5 for $c_0=3/2$.


Similarly, if we fix $b_1>0$ in case 1.3,  we obtain three phase portraits, G7, G8 and G9 in Figure \ref{fig:global}. If $b_3=b_1b_0/b_2$ then the connection of the separatrices in G8 takes place on the invariant line $y=-b_0/b_2$.
Fixed the values $c_0=2$, $b_1=b_2=b_3=1$ we get the phase portrait G7 for $b_0=1/2$, G8 for $b_0=1$ and G9 for $b_0=3/2$.

\begin{figure}[H] 
	\centering
	\scriptsize
	\stackunder[2pt]{\includegraphics[width=3cm]{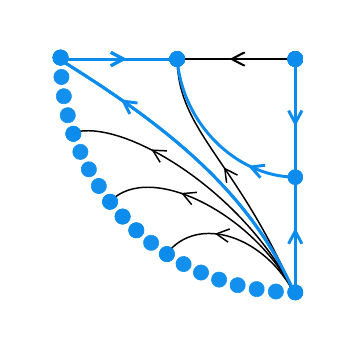}}{(a)}
	\stackunder[2pt]{\includegraphics[width=3cm]{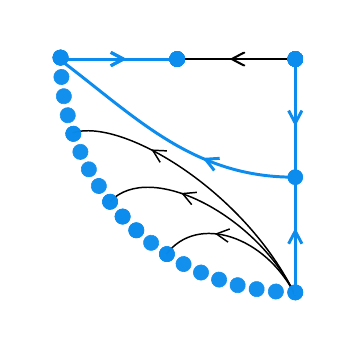}}{(b)}
	\stackunder[2pt]{\includegraphics[width=3cm]{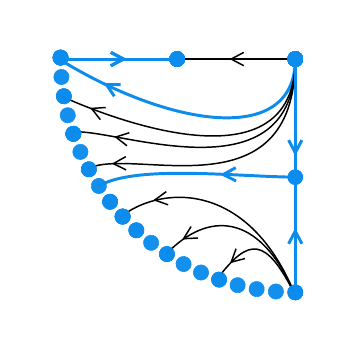}}{(c)}
	\caption{Possible configurations on the third quadrant on case 1.2 with $b_1>0$.}
	\label{fig:demsep}
\end{figure}

We include all global phase portraits obtained in Figure \ref{fig:global} and in Table \ref{tab:global} we indicate which of them are obtained in each case.

\begin{table}[H]
	\begin{center}
			\resizebox{0.7\textwidth}{!}{
	\begin{tabular}{|c|c|c|c|c|}
			\hline
			\small\textbf{Case}& 	\small\textbf{Subcase}& \small\textbf{$\boldsymbol{O_1}$}&  \small\textbf{$\boldsymbol{O_2}$} &  \small\textbf{Global} \\
			\hline
			\hline
			
			\multirow{2}{1cm}{\centering{1.1}}
			&
			$b_1>0$
			&
			\phantom{h} $L^1_2$ \phantom{h}
			&
			\phantom{h} $L^2_1$ \phantom{h}
			&
			G1 \\
			\cline{2-5}
			
			\multirow{2}{*}
			&
			$b_1<0$
			&
			$L^1_4$
			&
			$L^2_3$
			&
			G2\\
			\hline

			\multirow{2}{1cm}{\centering{1.2}}
			&
			$b_1>0$
			&
			$L^1_2$
			&
			$L^2_1$
			&
			G3, G4 or G5 \\
			\cline{2-5}
			
			\multirow{2}{*}
			&
			$b_1<0$
			&
			$L^1_4$
			&
			$L^2_3$
			&
			G6\\
			\hline
			
			\multirow{2}{1cm}{\centering{1.3}}
			&
			$b_1>0$
			&
			$L^1_1$
			&
			$L^2_2$
			&
			G7, G8 or G9 \\
			\cline{2-5}
			
			\multirow{2}{*}
			&
			$b_1<0$
			&
			$L^1_3$
			&
			$L^2_4$
			&
			G10\\
			\hline
			
			\multirow{2}{1cm}{\centering{1.4}}
			&
			$b_1>0$
			&
			$L^1_2$
			&
			$L^2_1$
			&
			G11 \\
			\cline{2-5}
			
			\multirow{2}{*}
			&
			$b_1<0$
			&
			$L^1_4$
			&
			$L^2_3$
			&
			G12\\
			\hline
			
			\multirow{2}{1cm}{\centering{1.5}}
			&
			$b_1>0$
			&
			$L^1_1$
			&
			$L^2_2$
			&
			G13 \\
			\cline{2-5}
			
			\multirow{2}{*}
			&
			$b_1<0$
			&
			$L^1_3$
			&
			$L^2_4$
			&
			G14\\
			\hline

			\multirow{2}{1cm}{\centering{2.1}}
			&
			$b_1>0$
			&
			$L^1_7$
			&
			$L^2_1$
			&
			G15 \\
			\cline{2-5}
			
			\multirow{2}{*}
			&
			$b_1<0$
			&
			$L^1_8$
			&
			$L^2_3$
			&
			G16\\
			\hline
			
			\multirow{2}{1cm}{\centering{2.2}}
			&
			$b_1>0$
			&
			$L^1_7$
			&
			$L^2_1$
			&
			G17 \\
			\cline{2-5}
			
			\multirow{2}{*}
			&
			$b_1<0$
			&
			$L^1_8$
			&
			$L^2_3$
			&
			G18\\
			\hline
			
			\multirow{2}{1cm}{\centering{2.3}}
			&
			$b_1>0$
			&
			$L^1_5$
			&
			$L^2_2$
			&
			G19 \\
			\cline{2-5}
			
			\multirow{2}{*}
			&
			$b_1<0$
			&
			$L^1_6$
			&
			$L^2_4$
			&
			G20\\
			\hline
			
			\multirow{2}{1cm}{\centering{2.4}}
			&
			$b_1>0$
			&
			$L^1_7$
			&
			$L^2_1$
			&
			G21 \\
			\cline{2-5}
			
			\multirow{2}{*}
			&
			$b_1<0$
			&
			$L^1_8$
			&
			$L^2_3$
			&
			G22\\
			\hline
			
			\multirow{2}{1cm}{\centering{3.1}}
			&
			$b_1>0$
			&
			$L^1_2$
			&
			$L^2_6$
			&
			G23 \\
			\cline{2-5}
			
			\multirow{2}{*}
			&
			$b_1<0$
			&
			$L^1_4$
			&
			$L^2_{8}$
			&
			G24\\
			\hline
			
			\multirow{2}{1cm}{\centering{3.2}}
			&
			$b_1>0$
			&
			$L^1_1$
			&
			$L^2_5$
			&
			G25 \\
			\cline{2-5}
			
			\multirow{2}{*}
			&
			$b_1<0$
			&
			$L^1_3$
			&
			$L^2_7$
			&
			G26\\
			\hline
			
			\multirow{2}{1cm}{\centering{3.3}}
			&
			$b_1>0$
			&
			$L^1_2$
			&
			$L^2_9$
			&
			G27 \\
			\cline{2-5}
			
			\multirow{2}{*}
			&
			$b_1<0$
			&
			$L^1_4$
			&
			$L^2_{10}$
			&
			G28\\
			\hline
			
			\multirow{2}{1cm}{\centering{3.4}}
			&
			$b_1>0$
			&
			$L^1_1$
			&
			$L^2_5$
			&
			G29 \\
			\cline{2-5}
			
			\multirow{2}{*}
			&
			$b_1<0$
			&
			$L^1_3$
			&
			$L^2_7$
			&
			G30\\
			\hline
			
			\multirow{2}{1cm}{\centering{4.1}}
			&
			$b_1>0$
			&
			$L^1_7$
			&
			$L^2_6$
			&
			G31\\
			\cline{2-5}
			
			\multirow{2}{*}
			&
			$b_1<0$
			&
			$L^1_8$
			&
			$L^2_{8}$
			&
			G32\\
			\hline
			
			\multirow{4}{1cm}{\centering{4.2}}
			&
			$b_1>0$, $c_0-b_0>0$
			&
			$L^1_5$
			&
			$L^2_5$
			&
			G33\\
			\cline{2-5}
			
			\multirow{4}{*}
			&
			$b_1>0$, $c_0-b_0<0$
			&
			$L^1_7$
			&
			$L^2_9$
			&
			G34\\
			\cline{2-5}
			
			\multirow{4}{*}
			&
			$b_1<0$, $c_0-b_0>0$
			&
			$L^1_6$
			&
			$L^2_7$
			&
			G35\\
			\cline{2-5}
			
			\multirow{4}{*}
			&
			$b_1<0$, $c_0-b_0<0$
			&
			$L^1_8$
			&
			$L^2_{10}$
			&
			G36\\
			\hline
		\end{tabular} }
			\caption{Classification of the global phase portraits of systems \eqref{sis2mu}.}
		\label{tab:global}
		\end{center}
\end{table}

\begin{figure}[H] 
	\centering
	\scriptsize
	\stackunder[2pt]{\includegraphics[width=3cm]{global/G1}}{(G1)}	\stackunder[2pt]{\includegraphics[width=3cm]{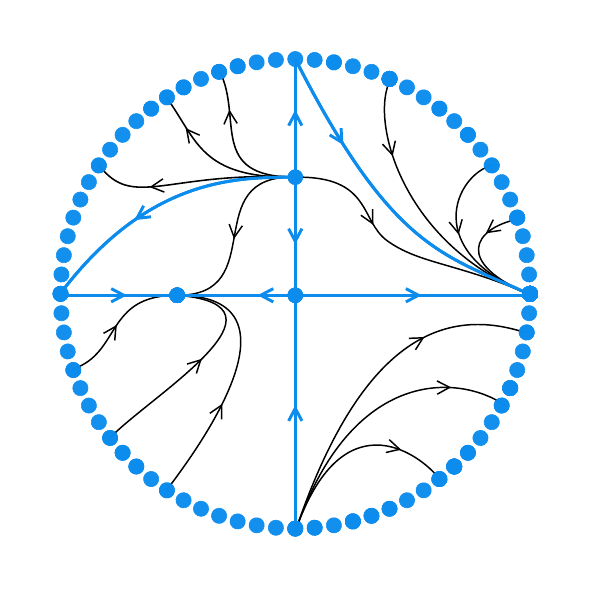}}{(G2)}
	\stackunder[2pt]{\includegraphics[width=3cm]{global/G3}}{(G3)}
	\stackunder[2pt]{\includegraphics[width=3cm]{global/G4}}{(G4)}
	\stackunder[2pt]{\includegraphics[width=3cm]{global/G5}}{(G5)}
	\stackunder[2pt]{\includegraphics[width=3cm]{global/G6}}{(G6)}
	\stackunder[2pt]{\includegraphics[width=3cm]{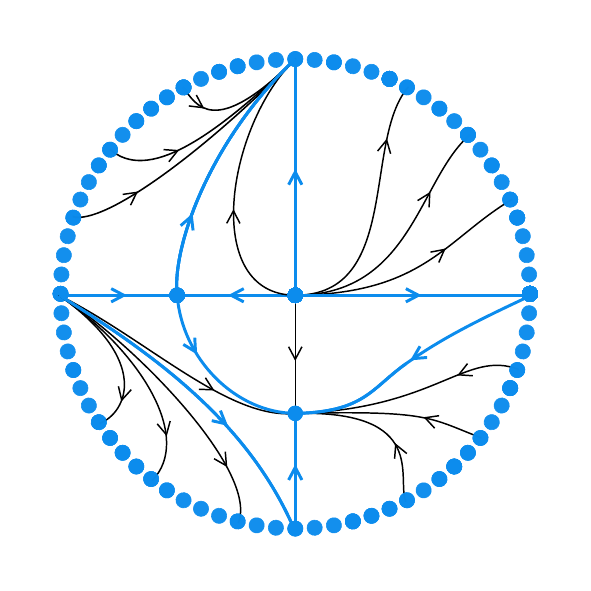}}{(G7)}
	\stackunder[2pt]{\includegraphics[width=3cm]{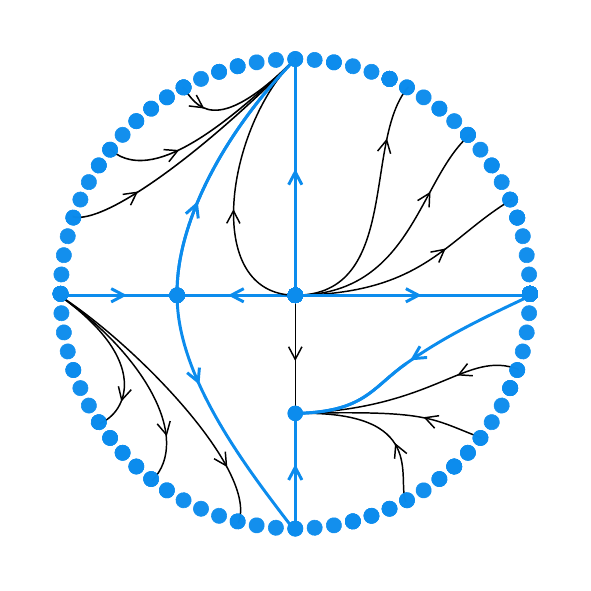}}{(G8)}
	\stackunder[2pt]{\includegraphics[width=3cm]{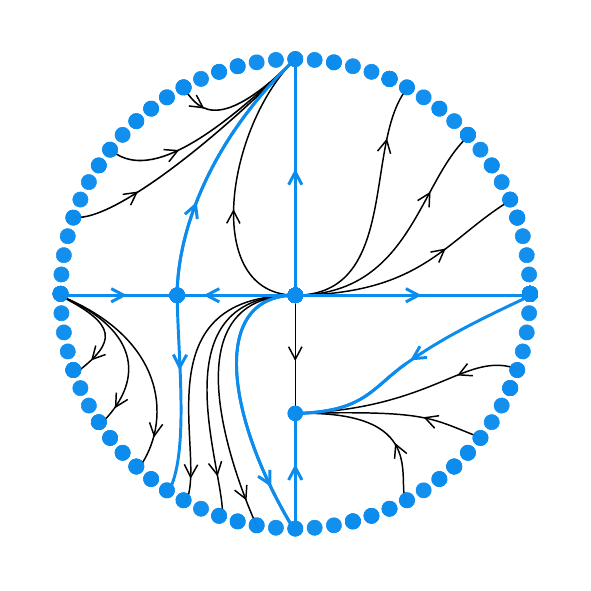}}{(G9)}
	\stackunder[2pt]{\includegraphics[width=3cm]{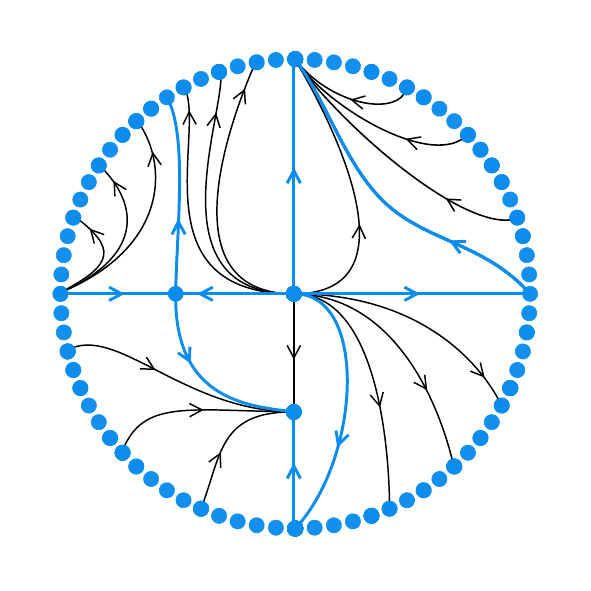}}{(G10)}
	\stackunder[2pt]{\includegraphics[width=3cm]{global/G11}}{(G11)}
	\stackunder[2pt]{\includegraphics[width=3cm]{global/G12}}{(G12)}
	\stackunder[2pt]{\includegraphics[width=3cm]{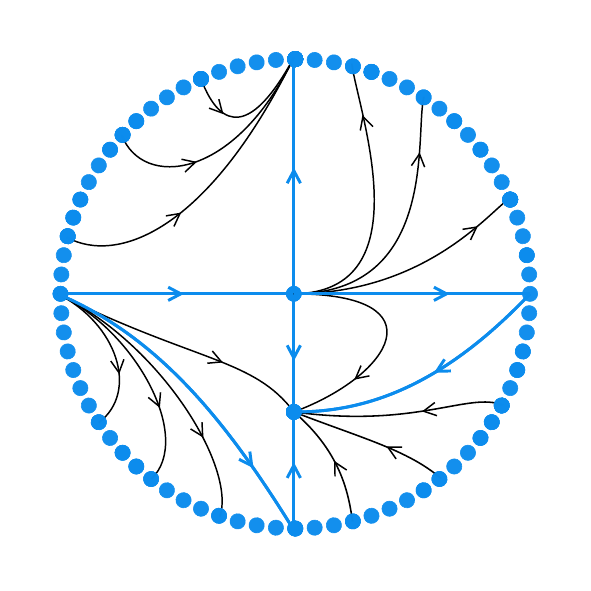}}{(G13)}
	\stackunder[2pt]{\includegraphics[width=3cm]{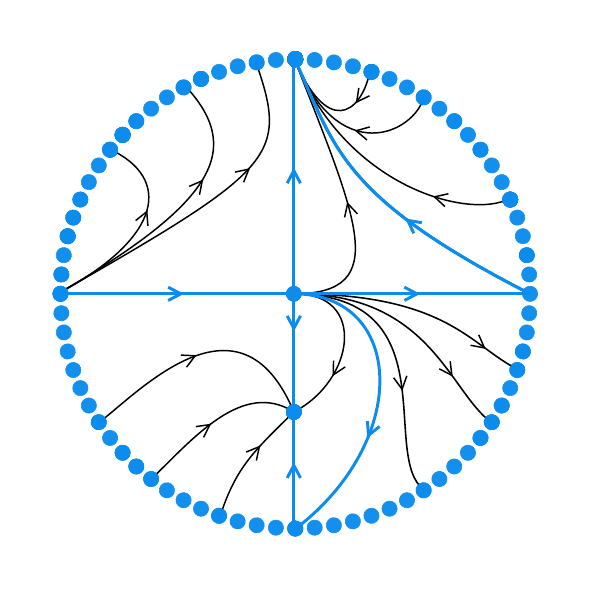}}{(G14)}
	\stackunder[2pt]{\includegraphics[width=3cm]{global/G15}}{(G15)}
	\stackunder[2pt]{\includegraphics[width=3cm]{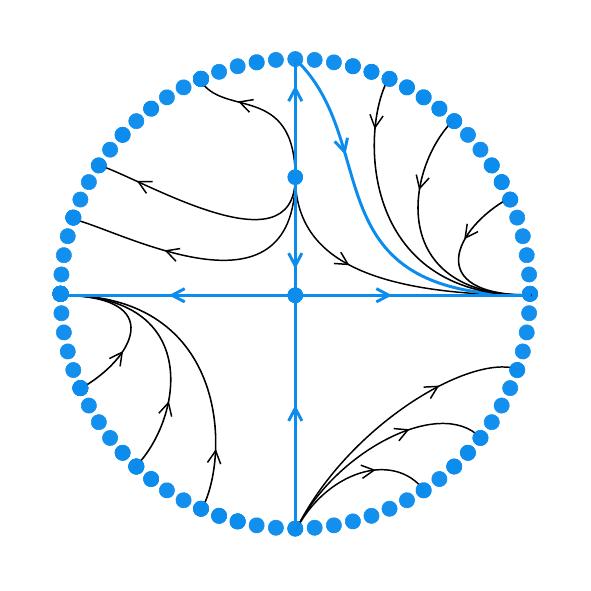}}{(G16)}
	\stackunder[2pt]{\includegraphics[width=3cm]{global/G17}}{(G17)}
	\stackunder[2pt]{\includegraphics[width=3cm]{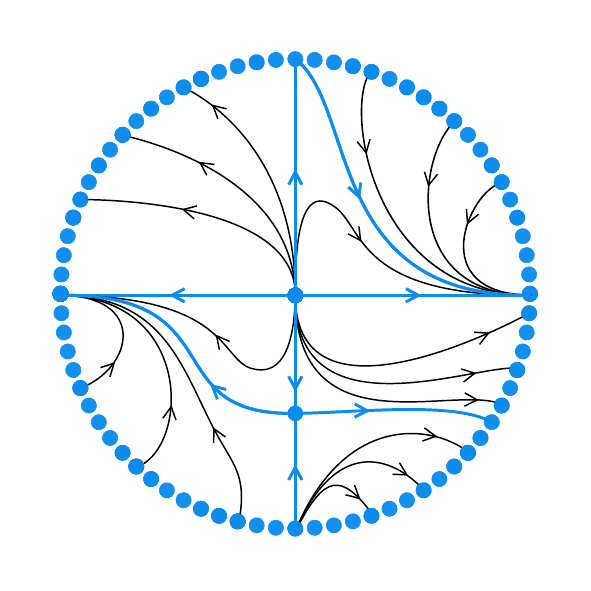}}{(G18)}
	\stackunder[2pt]{\includegraphics[width=3cm]{global/G19}}{(G19)}
	\stackunder[2pt]{\includegraphics[width=3cm]{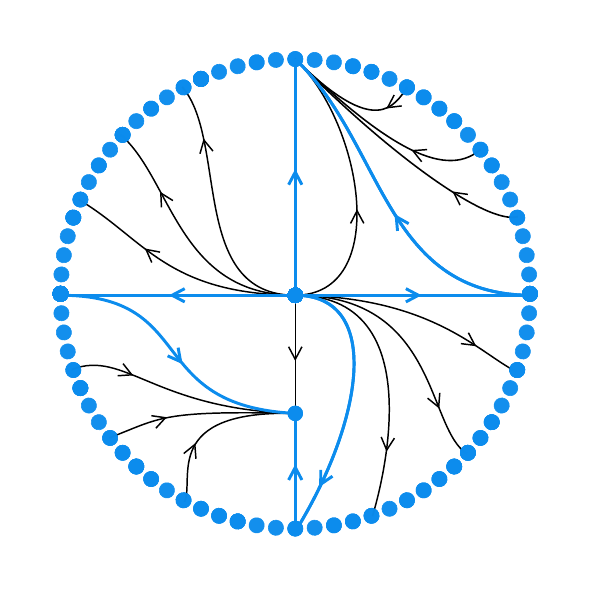}}{(G20)}
	\stackunder[2pt]{\includegraphics[width=3cm]{global/G21}}{(G21)}
	\stackunder[2pt]{\includegraphics[width=3cm]{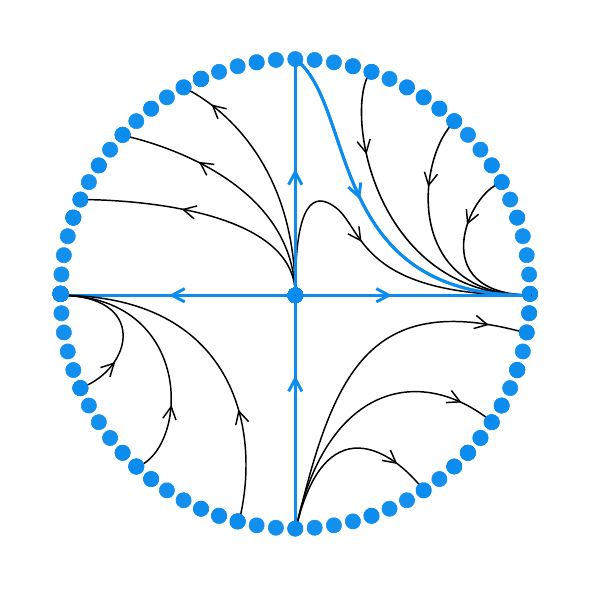}}{(G22)}
	\stackunder[2pt]{\includegraphics[width=3cm]{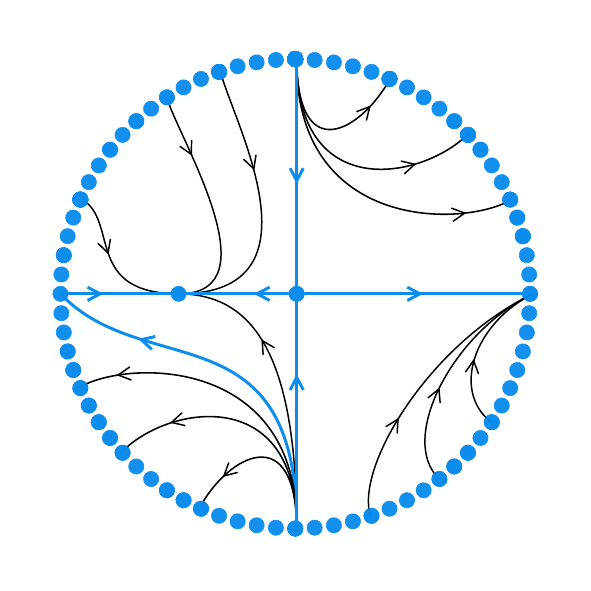}}{(G23)}
	\stackunder[2pt]{\includegraphics[width=3cm]{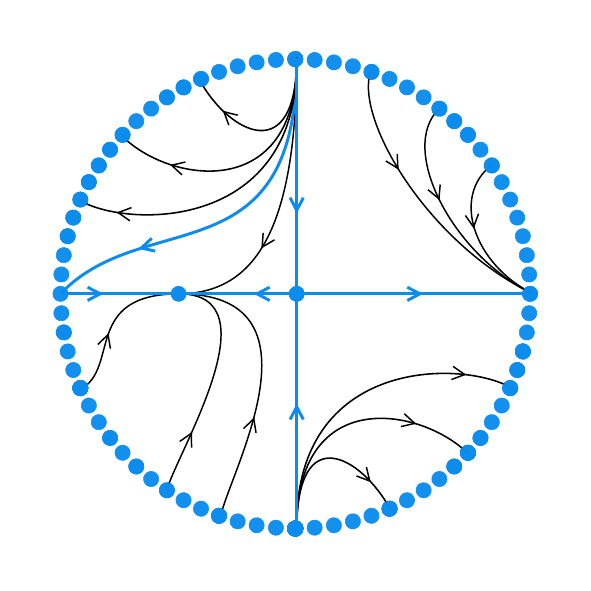}}{(G24)}
	\stackunder[2pt]{\includegraphics[width=3cm]{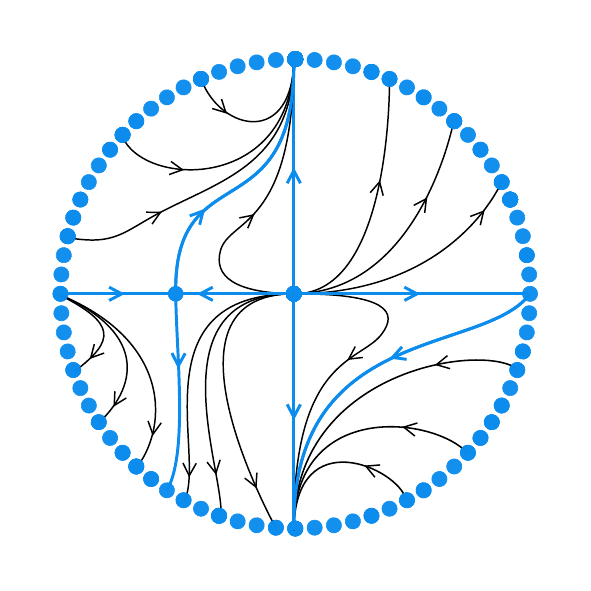}}{(G25)}
	\stackunder[2pt]{\includegraphics[width=3cm]{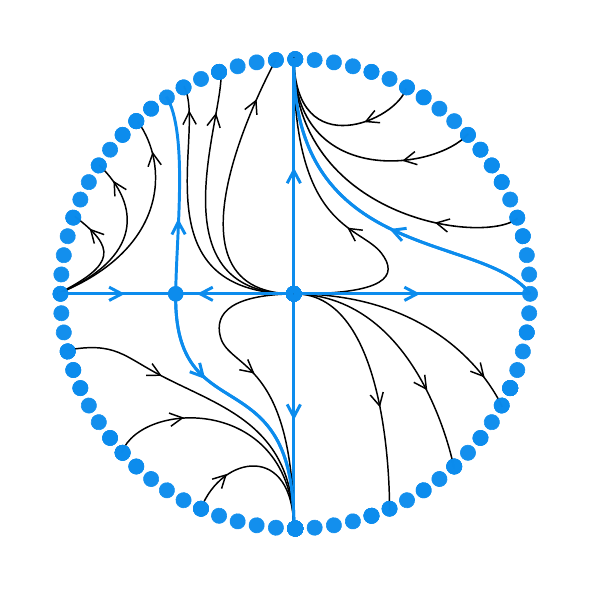}}{(G26)}
	\stackunder[2pt]{\includegraphics[width=3cm]{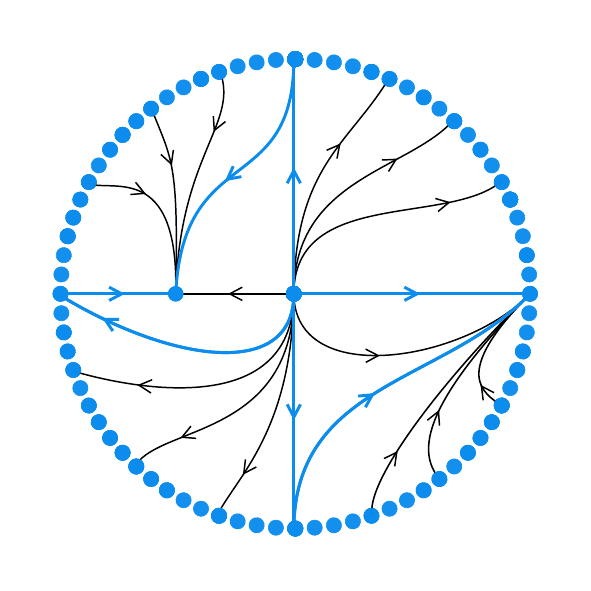}}{(G27)}
	\stackunder[2pt]{\includegraphics[width=3cm]{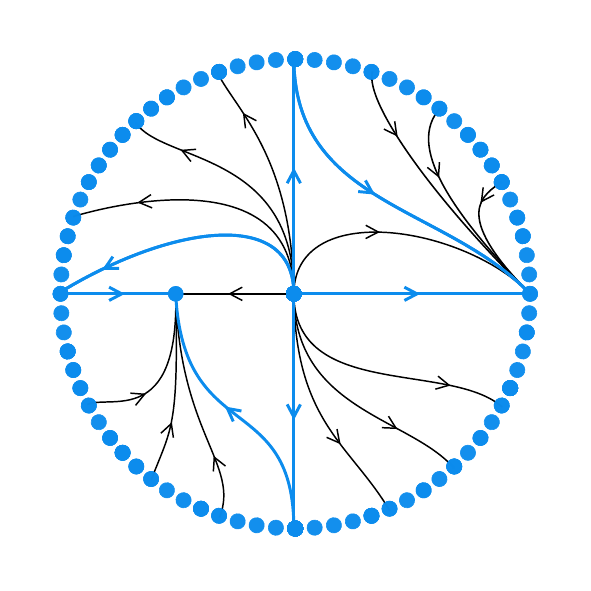}}{(G28)}
	\stackunder[2pt]{\includegraphics[width=3cm]{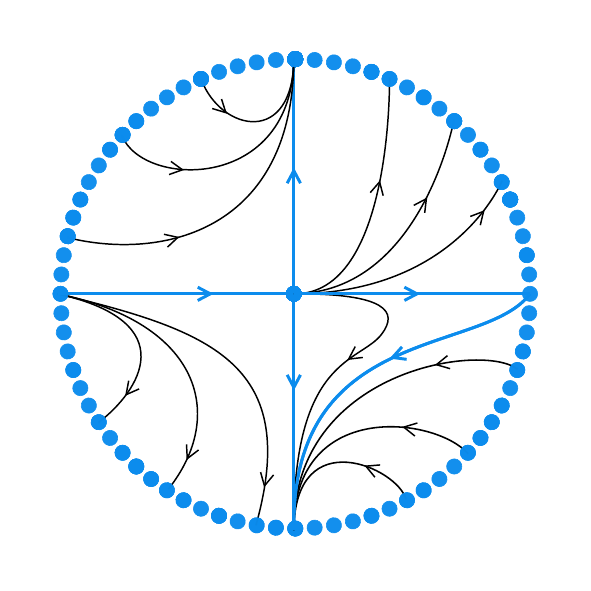}}{(G29)}
	\stackunder[2pt]{\includegraphics[width=3cm]{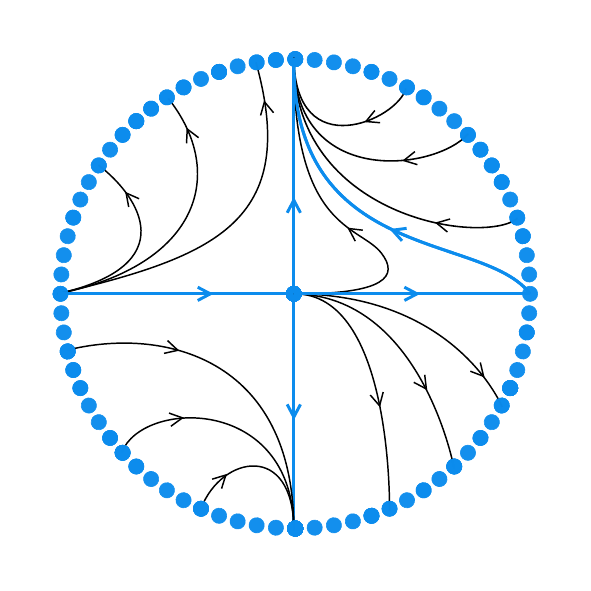}}{(G30)}
	\stackunder[2pt]{\includegraphics[width=3cm]{global/G31}}{(G31)}
	\stackunder[2pt]{\includegraphics[width=3cm]{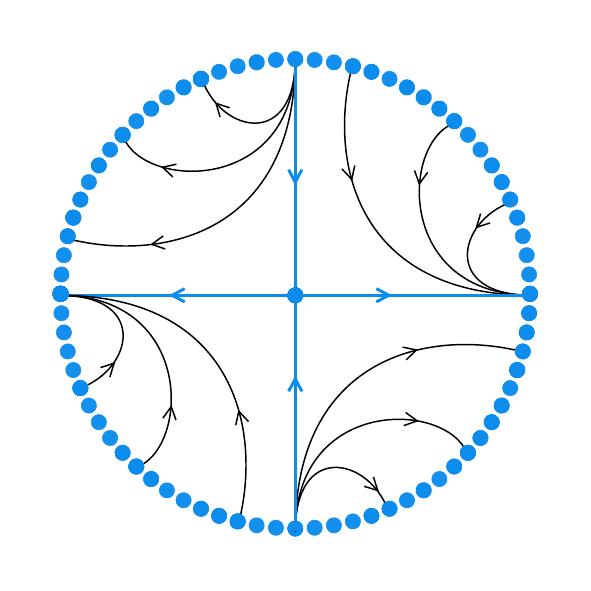}}{(G32)}
	\stackunder[2pt]{\includegraphics[width=3cm]{global/G33}}{(G33)}
	\stackunder[2pt]{\includegraphics[width=3cm]{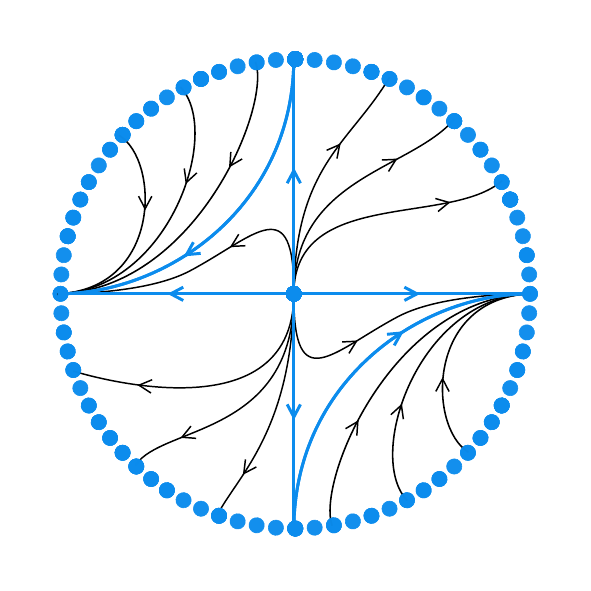}}{(G34)}
	\stackunder[2pt]{\includegraphics[width=3cm]{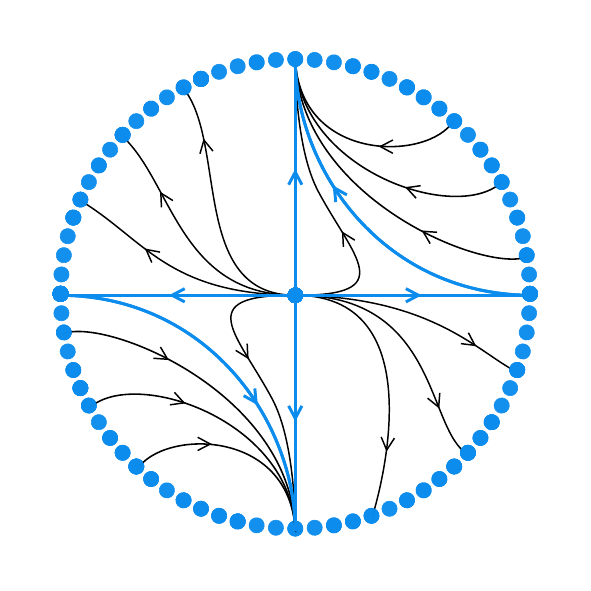}}{(G35)}
	\caption{Global phase portraits of systems \eqref{sis2mu} in the Poincar\'e disc.}
	\label{fig:global}
\end{figure}

We have obtained the 36 global phase portraits given in Figure \ref{fig:global} and now we study which of them are topologically equivalent. As Theorem \ref{th_MNP} only works in regions with a finite number of singular points, we will consider the equivalences on the open Poicaré disc, but this does not affect the result as if two separatrix configurations are topologically equivalent, they will be still equivalents if we add the boundary of the disc because the boundary is filled of singular points,and if they are not topologically equivalent they will not be equivalent by adding the boundary of the disc.

We will consider classes of equivalence according to the following invariants: the number of finite singular points and the sum of the indices at the finite singular points, denoted by $ind_F$. We give this first classification in Table \ref{tab:clases} and then within each class we prove which of the phase portraits are topologically equivalent. 

\begin{table}[h]
	\begin{tabular}{|c|c|c|l|}
			\hline
			\textbf{Class}& \textbf{Nº finite singular points} & \textbf{$\boldsymbol{ind_F}$} & \textbf{Global phase portraits}  \\
			\hline
			\hline
			1
			&
			3
			&
			1
			&
			G1, G2, G3, G4, G5, G6, G7, G8, G9, G10.
			\\
			\hline
			
			2
			&
			\multirow{3}{1cm}{\centering{2}}
			&
			1
			&
			G11, G12, G13, G14.
			\\
			\cline{1-1}\cline{3-4}

			3
			&
			\multirow{3}{*}
			&
			0
			&
			G15, G16, G17, G18, G23, G24, G25, G26.
			\\
			\cline{1-1}\cline{3-4}

			4
			&
			\multirow{3}{*}
			&
			2
			&
			G19, G20, G27, G28.
			\\
			\hline

			5
			&
			\multirow{3}{1cm}{\centering{1}}
			&
			0
			&
			G21, G22, G29, G30.
			\\
			\cline{1-1}\cline{3-4}

			6
			&
			\multirow{3}{*}
			&
			-1
			&
			G31, G32.
			\\
			\cline{1-1}\cline{3-4}

			7
			&
			\multirow{3}{*}
			&
			1
			&
			G33, G34, G35, G36.
			\\
			\hline

		\end{tabular}
		\caption{Classes of equivalence according to the number of finite singular points and to the $ind_F$. } 
		\label{tab:clases}
\end{table}

\textbf{Class 1.} First we can distiguish two subclasses depending on the number of separatrices in the open Poincaré disc. There are 11 separatrices in phase portraits G1, G2, G4 and G8, and 12 separatrices in the phase portraits G3, G5, G6, G7, G9 and G10. In the first subclass, G1 is topologically equivalent to G2 by doing a symmetry with respect to the line $z=-y$ and a change of the time variable $t$ by $-t$. G1 is different from G4 as in G1 there are two separatrices that start in the unstable node and in G4 there are three. 
G1 is also different from G8 as in G1 there are two separatrices of the saddle that connect with the infinity and in G8 there are three. At last G4 is topologically equivalent to G8 by doing a 90º rotation of G8 and then a symmetry with respect to the $z$-axis. In the second subclass, G3 is different from G5 as in G3 the saddle has two separatrices that connect with the infinity and in G5 it has three. By doing a symmetry with respect to the line $y=z$ we transform G3 intro G7, G5 into G9 and G6 into G10. G7 is different from G10 as in G7 there are three separatrices that start in the unstable node and in G10 there are four. G9 is different from G10 as in G10 there is a separatrix that connects two infinite singular points but in G9 there is not a such separatrix.

\vspace{0.2cm}

\textbf{Class 2.} G11 is different from G12 as in G11 the saddle-node has three separatrices that connect with infinite singular points and in G12 it has four. G11 is topologically equivalent to G13 and G12 to G14 by doing a symmetry with respect to the line $y=z$. 

\vspace{0.4cm}

\textbf{Class 3.} G15 is topologically equivalent to G16 and G17 to G18 by doing a symmetry with respect to the $z$-axis. G15 is different from G17 as in G15 there are two separatrices  that start at the node and in G17 there are four. G15 is topologically equivalent to G23 by doing a rotation of 90º in G15 and a change of the time variable $t$ by $-t$. We also can transform G25 into G18 by a rotation of 90º. Lastly we can transform G23 into G24 and G25 into G26 with a symmetry with respect to the $y$-axis.

\vspace{0.2cm}

\textbf{Class 4.} G19 is topologically equivalent to G20 by a symmetry with respect to the $z$-axis, G19 to G27 by a symmetry with respect to the line $z=y$ and G27 to G28 by a symmetry with respect to $y$-axis.

\vspace{0.2cm}

\textbf{Class 5.} G21 is topologically equivalent to G22 by a symmetry with respect to the $z$-axis, G21 to G29 by a symmetry with respect to the line $z=y$ and G29 to G30 by a symmetry with respect to $y$-axis.

\vspace{0.2cm}

\textbf{Class 6.} G31 is topologically equivalent to G32 by a symmetry with respect to the $z$-axis.

\vspace{0.2cm}

\textbf{Class 7.} G33 is topologically equivalent to G34 with a symmetry with respect to the line $z=y$, and by a symmetry with respect to the $z$-axis G33 is topologically equivalen to G35 and G34 to G36.

\vspace{0.2cm}

In summary, among these seven clases, we have found 13 topologically different phase portraits in the Poincaré disc for systems \eqref{sis2mu},so we have proved Theorem \ref{th_global}. These 13 distinct phase portraits are described in Figure \ref{fig:global_top}, where we include a representative of each one of the topological equivalence classes. These representatives correspond with the phase portraits in Figure \ref{fig:global} as follows:

\begin{table}[H]
	\centering
		\resizebox{0.3\textwidth}{!}{
	\begin{tabular}[t]{ |c | l |}
		\hline
		\textbf{Rep.}      & 	\textbf{Phase portraits }   \\ \hline  \hline
		R1   & G1, G2.  \\ \hline
		R2   & G3, G7.\\ \hline
		R3  & G4, G8. \\ \hline
		R4  & G5, G9.  \\ \hline
		R5 &  G6, G10 \\
		\hline
	\end{tabular} }\hfill%
	\resizebox{0.3\textwidth}{!}{
	\begin{tabular}[t]{ |c | l |}
		\hline
		\textbf{Rep.}        & 	\textbf{Phase portraits}     \\ \hline  \hline
		R6 &  G11, G13. \\ \hline
		R7 & G12, G14.  \\ \hline
		R8 & G15,G16, G23, G24.  \\ \hline
		R9 & G17, G18, G25, G26.  \\ 
		\hline
	\end{tabular}}\hfill%
	\resizebox{0.33\textwidth}{!}{
	\begin{tabular}[t]{| c | l |}
		\hline
		\textbf{Rep.  }     & 	\textbf{Phase portraits  }   \\ \hline  \hline
		R10 & G19, G20, G27, G28. \\ \hline
		R11 & G21, G22, G29, G30.  \\ \hline
		R12 & G31, G32.  \\ \hline
		R13 &  G33, G34, G35, G36. \\ 
		\hline
	\end{tabular} }\hfill%
	\caption{Representatives of each equivalence class and their corresponding global phase portraits of systems \eqref{sis2mu}.}
	\label{tab:rep}
\end{table}

Now we give a second proof that shows that these 13 phase portraits are indeed topologically distinct. 

\begin{theorem}
	The 13 phase portraits of systems \eqref{sis2mu} included in Figure \ref{fig:global_top} are topologically distinct.
\end{theorem}

\begin{proof}
	We will consider six geometrical invariants in order to distinguish the phase portraits.
	
	\begin{enumerate}
		
		\item[$(I_1)$] Number of finite singularities. The values of this invariant for the phase portraits R1 to R13 are, respectively: (3, 3, 3, 3, 3, 2, 2, 2, 2, 2, 1, 1, 1).

		\vspace{0.1cm}
		
		\item[$(I_2)$] Sum of the index of the finite singularities. The values of this invariant for the phase portraits R1 to R13  are (1, 1, 1, 1, 1, 1, 1, 0, 0, 2, 0, -1, 1). 
		
	\end{enumerate}
	
	With these two invariants we can already determine that the phase portraits R10, 
	R11, R12 and R13 are topologically distinct between them and from all the others. We will not determine other invariants for them.

	\begin{enumerate}
		
		\item[$(I_3)$] Separatrices of the finite singularities connected with finite nodes. For the phase portraits R1 to R9 this invariant has the values:  (2, 2, 1, 1, 2, 1, 1, 1, 1).

		\vspace{0.1cm}

		\item[$(I_4)$]  Number of connections between separatrices of the finite singularities and separatrices of infinite singularities. The values of this invariant for the phase portraits R1 to R9 are: (2, 1, 2, 1, 1, 2, 2, 3, 1). 
		
	\end{enumerate}
	
	With the four previous invariants we can guarantee that the phase portraits R1, R3, R4, R8 and R9 are topologically distinct. Among the remaining phase portraits, R2 has the same invariants as R5 and R6 the same as R7, so we will distinguish between them with the two following invariants.

	\begin{enumerate}
		
		\item[$(I_5)$] Number of infinite singularities receiving an infinite number of orbits from a finite singularity. This invariant is 2 for R2 and 1 for R5.
		
		\vspace{0.1cm}
		
		\item[$(I_6)$]  Number of separatrices that leave the finite saddle-node in its parabolic sector and go to an infinite singular point. This invariant is 1 for R6 and 2 for R7.
		
	\end{enumerate}
	
	Then we have proved that all the 13 phase portraits are topologically distinct as they have different values for the mentioned invariants.
	
\end{proof}

\section*{Acknowledgements}

We thank to the reviewer his/her comments and suggestions which help us to improve the presentation of our results

The first and third authors are partially supported by the Ministerio de Ciencia e Innovación, Agencia Estatal de Investigación (Spain), grant PID2020-115155GB-I00 and the Consellería de Educación, Universidade e Formación Profesional (Xunta de Galicia), grant ED431C 2019/10 with FEDER funds. The first author is also supported by the Ministerio de Educacion, Cultura y Deporte de España, contract FPU17/02125.

The second author is partially supported by the Ministerio de Ciencia, Innovaci\'on y Universidades, Agencia Estatal de Investigaci\'on grant PID2019-104658GB-I00, the Ag\`encia de Gesti\'o d'Ajuts Universitaris i de Recerca grant 2017SGR1617, and the H2020 European Research Council grant MSCA-RISE-2017-777911.

\end{document}